\documentclass[a4paper, 11pt]{article}

\usepackage{amsmath}
\usepackage{amsthm}
\usepackage{amsfonts}
\usepackage{amssymb}
\usepackage{amscd}
\usepackage[dvips]{graphicx, color}

\newtheorem{thm}{Theorem}[section]
\newtheorem{cor}[thm]{Corollary}

\newtheorem{lem}[thm]{Lemma}
\newtheorem{prop}[thm]{Proposition}
\newtheorem{rem}[thm]{Remark}
\newtheorem{exam}[thm]{Example}

\numberwithin{equation}{section}

\begin{document}
\title{\bf \Large A characterization of Conway-Coxeter friezes of zigzag type by rational links}
\author{
{Takeyoshi Kogiso}\\
{\footnotesize  Department of Mathematics, Josai University, }\\
{\footnotesize  1-1, Keyakidai Sakado, Saitama, 350-0295, Japan}\\
{\footnotesize  E-mail address: kogiso@josai.ac.jp}\\  
\\ 
{Michihisa Wakui}\\ 
{\footnotesize  Department of Mathematics, Faculty of Engineering Science,}\\ 
{\footnotesize  Kansai University, Suita-shi, Osaka 564-8680, Japan}\\
{\footnotesize  E-mail address: wakui@kansai-u.ac.jp}\\ 
}

\maketitle 
\begin{abstract}
The present paper show that Conway-Coxeter friezes of zigzag type are characterized by (unoriented) rational links. 
As an application of this characterization Jones polynomial can be defined for Conway-Coxeter friezes of zigzag type. 
This gives a new method for computing the Jones polynomial for oriented rational links. 
\end{abstract}

\baselineskip 16pt 
In the precedent research \cite{Kogiso-Wakui} the authors found that some beautiful relation between frieze patterns due to Conway and Coxeter \cite{Coxeter, CoCo1, CoCo2} and rational link diagrams. 
In that paper it is shown that Conway-Coxeter friezes of zigzag type are determined by $4$-tuple of rational numbers which are related by $\frac{p}{M}, \frac{q}{M}, \frac{r}{M}, \frac{s}{M}$, where $p, q, r, s, M$ are positive integers such that $p+q=M=r+s$. 
It is noteworthy that the numerators $p, q, r, s$ are located around $M$ in the corresponding Conway-Coxeter frieze to $\frac{p}{M}$.  
Moreover, it is shown in \cite{Kogiso-Wakui} that the Kauffman bracket polynomials of the rational link diagrams corresponding to $\frac{p}{M}, \frac{q}{M}, \frac{r}{M}, \frac{s}{M}$ coincide up to replacing $A$ with $A^{-1}$. 
This implies that the Kauffman bracket polynomials have meaning for Conway-Coxeter friezes of zigzag type. 
In the present paper we develop this consideration, and show that Conway-Coxeter friezes of zigzag type are characterized by (unoriented) rational links. 
We derive this characterization by reformulating the classification result on (unoriented) rational links by Schubert \cite{Schubert} 
in terms of the operations $i(\frac{p}{M})=\frac{q}{M},\ r(\frac{p}{M})=\frac{r}{M},\ (ir)(\frac{p}{M})=\frac{s}{M}$, which are introduced in \cite{Kogiso-Wakui}. 
\par 
Any non-zero rational number is classified into three types such as $\frac{0}{1}, \frac{1}{0}, \frac{1}{1}$, that are determined by the parities of its numerator and denominator. 
We study in detail on types of non-zero rational numbers by language of Farey sums and continued fraction expansions. 
By using this result and applying the above characterization of Conway-Coxeter friezes of zigzag type, 
we introduce Jones polynomials for the Conway-Coxeter friezes of zigzag type. 
The key is to know difference between modified writhes of the four rational link diagrams corresponding to $\frac{p}{M}, \frac{q}{M}, \frac{r}{M}, \frac{s}{M}$. 
Recently, Nagai and Terashima \cite{Nagai_Terashima} found a combinatorial formula to compute the writhes for oriented rational link diagrams  in terms of continued fraction expansion.  
Their formula is described by some sign sequence determined from ancestor triangles of rational numbers, 
which are introduced by Hatcher and Ortel~\cite{HO} or S. Yamada~\cite{Yamada-Proceeding} in different backgrounds. 
We give a recursive formula for computing the sign sequence without geometric picture. 
Thus, the Jones polynomials for the Conway-Coxeter friezes of zigzag type can be computed from continued fraction expansions of rational numbers in a completely combinatorial way. 
\par
This also gives a new method for computing the Jones polynomial for oriented rational links. 
Kyungyon Lee and Ralf Schiffler \cite{LeeSchiffler} gave an interesting formula to express Jones polynomials for rational links as specializations for cluster variables by using snake graphs. 
On the other hand, Sophie Morier-Genoud and Valentin Ovsienko~\cite{M-GO2} introduced a notion of $q$-deformed rational numbers and $q$-deformed continued fractions and applied to calculate normalized $q$-Jones polynomials as other approach of Lee and Schiffler's result.
Our formula on the Jones polynomials for the Conway-Coxeter friezes of zigzag type also gives a new method for computing the Jones polynomial for oriented rational links. 
\par 
The present paper is organized as follows. 
In Section 1 we recall some notations of continued fraction expansions for rational numbers, Farey sums and $LR$ words. 
Three important operations $i, r, ir$ on $LR$ words or equivalently on the rational numbers in the open interval $(0,1)$ are introduced. 
In Section 2 we briefly explain the definition of the Conway-Coxeter friezes of zigzag type, and effect of the operations $i,r, ir$ on such friezes. 
In Section 3 we show that any Conway-Coxeter frieze of zigzag type can be regarded as an unoriented rational link. 
In Section 4, as an application of the result in Section 3, ``Jones polynomial" for the Conway-Coxeter friezes of zigzag type can be defined. 
Furthermore, it is observed that 
each pair of rational knots such that their Jones polynomials are the same up to replacing $t$ with $t^{-1}$ 
has some common characteristic, which would be considered as a new phenomena. 
In the final section we derive a recurrence formula for computing the sign sequence which is used in the writhe formula of a rational link diagram due to Nagai and Terashima~\cite{Nagai_Terashima}. 
\par 
Throughout of the present paper, $\mathbb{N}$ denotes the set of positive integers. 
On (rational) tangles, knots and links and their diagrams we refer the reader to Cromwell's Book~\cite{Cromwell} and Murasugi's Book~\cite{Murasugi}. 

\section{Continued fraction expansions, Farey sums and $LR$ words}
\par 
In the present paper, the denominator $q$ of any irreducible fraction $\frac{p}{q}$ is always assumed to be $q\geq 0$, and if $q=0$, then $p=1$. 
\par 
Two irreducible fractions $\frac{p}{q}$ and $\frac{r}{s}$ are said to be \textit{Farey neighbors} if they satisfy $qr-ps=1$. 
Then $\frac{p}{q}<\frac{r}{s}$ holds, 
$\frac{p}{q}\sharp \frac{r}{s}:=\frac{p+r}{q+s}$ 
is also irreducible, and both $\frac{p}{q}, \frac{p}{q}\sharp \frac{r}{s}$ and $\frac{p}{q}\sharp \frac{r}{s}, \frac{r}{s}$ are Farey neighbors, again. 
It is well-known that $\frac{p}{q}\sharp \frac{r}{s}$ is the unique fraction that the absolute values of the numerator and the denominator are minimum between the numerators and the denominators of irreducible fractions in the open interval $\bigl(\frac{p}{q}, \frac{r}{s}\bigr)$, respectively. 
It can be also verified that for any nonzero rational number $\alpha $, there is a unique pair $\bigl(\frac{p}{q}, \frac{r}{s}\bigr)$ of Farey neighbors which satisfies $\alpha =\frac{p}{q}\sharp \frac{r}{s}$. 
The pair $\bigl(\frac{p}{q}, \frac{r}{s}\bigr)$ is called the parents of $\alpha $, and $\alpha $ is called the mediant of $\bigl(\frac{p}{q}, \frac{r}{s}\bigr)$. 
\par 
There is a one-to-one correspondence between the rational numbers in the open interval $(0,1)$ and the $LR$ words as explained below. 
We denote the corresponding $LR$ word by $w(\alpha )$ for a rational number $\alpha \in (0,1)$. 
Then, the function $w(\alpha )$ is given by the following recurrence formula \cite[Lemma 3.2]{Kogiso-Wakui}: 
\begin{enumerate}
\item[$\bullet$] $w\bigl(\frac{1}{2}\bigr)=\emptyset $, 
\item[$\bullet$] $w\bigl(\frac{p}{q}\sharp \frac{r}{s}\bigr)=\begin{cases}
Lw\bigl(\frac{r}{s}\bigr) & \text{if}\ q<s, \\ 
Rw\bigl(\frac{p}{q}\bigr) & \text{if}\ q>s. 
\end{cases}$
\end{enumerate}

\par \smallskip 
\begin{exam}
\begin{enumerate}
\item[$(1)$] $w\bigl(\frac{1}{3}\bigr)
=w\bigl(\frac{0}{1}\sharp \frac{1}{2}\bigr)
=Lw\bigl(\frac{1}{2}\bigr) =L$ and 
$w\bigl(\frac{1}{4}\bigr)
=w\bigl(\frac{0}{1}\sharp \frac{1}{3}\bigr)
=Lw\bigl(\frac{1}{3}\bigr) =L^2$. 
In general, $w(\frac{1}{n})
=L^{n-2}$ for an integer $n \geq 3$. 
\par 
\item[$(2)$] $w\bigl(\frac{2}{3}\bigr)
=w\bigl(\frac{1}{2}\sharp \frac{1}{1}\bigr)
=Rw\bigl(\frac{1}{2}\bigr) =R$ and 
$w\bigl(\frac{3}{4}\bigr)
=w\bigl(\frac{1}{2}\sharp \frac{2}{3}\bigr)
=Rw\bigl(\frac{2}{3}\bigr) =R^2$. 
In general, $w\bigl(\frac{n-1}{n}\bigr) 
=R^{n-2}$ for an integer $n \geq 3$. 
\end{enumerate}
\end{exam}

\begin{rem}
As explained in \cite{Kogiso-Wakui}, the above correspondence can be visualized by using the Stern-Brocot tree as follows. 
Let $\alpha \in \mathbb{Q}\cap (0,1)$. 
In the Stern-Brocot tree, starting from the vertex $\frac{1}{2}$ we record $L$ or $R$ according to the left down or right down until reaching 
the vertex $\alpha $ along edges, and arrange the sequence of $L$ and $R$ in the direction from right to left. 
This sequence coincides with $w(\alpha )$. 
\end{rem} 

\par \smallskip 
For an $LR$ word $w$ we denote by $i(w)$ the word obtained by changing $L$ with $R$, and by $r(w)$ the word obtained by reversing the order. We also consider the word $(ir)(w)$ obtained by composing $i$ and $r$. 
Then we have the following. 

\par \smallskip 
\begin{lem}\label{1-2}
Let $\alpha =\frac{p}{q}\in \mathbb{Q}\cap (0,1)$ be an irreducible fraction, and $\bigl( \frac{x}{r}, \frac{y}{s}\bigr)$ be the pair of parents of $\alpha $. 
If $w=w\bigl( \frac{p}{q}\bigr)$, 
then 
\begin{enumerate}
\item[$(1)$] $i(w)= w\bigl( \frac{q-p}{q}\bigr)$, 
\item[$(2)$] $r(w)= w\bigl( \frac{r}{q}\bigr)$, 
\item[$(3)$] $(ir)(w)= w\bigl( \frac{s}{q}\bigr)$. 
\end{enumerate}

Based on the above result, $i(\alpha ), r(\alpha ), (ir)(\alpha )$ are defined by 
$i(\alpha )=\frac{q-p}{q}, 
r(\alpha )=\frac{r}{q}, 
(ir)(\alpha )=\frac{s}{q}$, respectively. 
\end{lem} 

\par \smallskip 
For the proof of the above lemma see \cite[Lemmas 3.5 and 3.8]{Kogiso-Wakui}. 
\par 
Let $\alpha $ be a rational number, and 
expand it as a continued fraction  
\begin{equation}\label{eqw1-10}
\begin{aligned}
\alpha =a_0+\dfrac{1}{\vbox to 18pt{ }a_1+\dfrac{1}{\vbox to 18pt{ }a_2+\dfrac{1}{\vbox to 18pt{ }\ddots +\dfrac{1}{\vbox to 18pt{ }a_{n-1}+\dfrac{1}{a_n}}}}}, 
\end{aligned}
\end{equation}

\noindent 
where $a_0\in \mathbb{Z},\ a_1, a_2, \ldots , a_n\in \mathbb{N}$. 
We denote the right-hand side of \eqref{eqw1-10} by $[a_0, a_1, a_2, \ldots , a_n]$. 
Note that the expansion \eqref{eqw1-10} is unique if the parity of $n$ is specified. 

\par \smallskip 
\begin{lem}\label{1-4}
For a rational number $\alpha =[0, a_1, a_2, \ldots , a_n]$ in the open interval $(0,1)$ 
\begin{enumerate}
\item[$(1)$] $i(\alpha )=[0, 1, a_1-1, a_2, \ldots , a_n]$. 
\item[$(2)$] If $n$ is even, then $r(\alpha )=[0, 1, a_n-1, a_{n-1}, \ldots , a_2, a_1]$ and $(ir)(\alpha )=[0, a_n, \ldots , a_2, a_1]$. 
\end{enumerate}
\end{lem}

Proofs of the above lemma can be found in \cite[Corollary 3.7 and Lemma 3.11]{Kogiso-Wakui}. 

\par 
The rational numbers are classified into three types as follows. 
Let $\alpha $ be a rational number, and express it as an irreducible fraction $\alpha =\frac{p}{q}$.  
We call $\alpha $ $\frac{1}{1}$-type if $p\equiv 1,\ q\equiv 1\ (\text{mod}\ 2)$. 
Similarly, $\frac{1}{0}$-type and $\frac{0}{1}$-type are defined. 
Let us define $n(\alpha ), d(\alpha )\in \{ 0,1\}$ by the equations $p\equiv n(\alpha ),\ q\equiv d(\alpha )\ (\text{mod}\ 2)$. 
When $\alpha $ is in the open interval $(0,1)$ and is expressed in the continued fraction form $\alpha =[0, a_1, \ldots , a_n]$ with $n\geq 3$, 
we set $\alpha _0:=0,\ \alpha _i=[0, a_1, \ldots , a_i]$ for all $i=1, \ldots , n$. 
Then the following recurrence equations hold for all $i\geq 2$:
\begin{align*}
n(\alpha _i)&=n(\alpha _{i-2})+\dfrac{1-(-1)^{a_i}}{2}n(\alpha _{i-1}), \\ 
d(\alpha _i)&=d(\alpha _{i-2})+\dfrac{1-(-1)^{a_i}}{2}d(\alpha _{i-1}), 
\end{align*}
where these equations are treated in modulo $2$.

\par \bigskip \noindent 
\begin{lem}\label{1-5}
Let $\alpha =\frac{p}{q}$ be an irreducible fraction in $(0,1)$, and let $\bigl( \frac{x}{r}, \frac{y}{s}\bigr)$ be the pair of parents. 
We write in the continued fraction form $\alpha =[0, a_1, \ldots , a_n]$ for an even number $n$, and let 
$N_0(a_1, \ldots , a_n)$ be the number of even integers in $a_1, \ldots , a_n$. 
\begin{enumerate}
\item[$(1)$] If $\alpha $ is $\frac{1}{0}$-type, then 
\begin{enumerate}
\item[$\bullet$] $N_0(a_1, \ldots , a_n)$ is even if and only if $x$ is even, 
\item[$\bullet$] $N_0(a_1, \ldots , a_n)$ is odd  if and only if $y$ is even. 
\end{enumerate}
\item[$(2)$] If $\alpha $ is $\frac{1}{1}$-type, then 
\begin{enumerate}
\item[$\bullet$] $N_0(a_1, \ldots , a_n)$ is even  if and only if $y$ is even, 
\item[$\bullet$] $N_0(a_1, \ldots , a_n)$ is odd if and only if  $x$ is even. 
\end{enumerate}
\item[$(3)$] If $\alpha $ is $\frac{0}{1}$-type, then 
$x, y$ are odd, and 
\begin{enumerate}
\item[$\bullet$] $N_0(a_1, \ldots , a_n)$ is even  if and only if $s$ is even, 
\item[$\bullet$] $N_0(a_1, \ldots , a_n)$ is odd  if and only if $r$ is even. 
\end{enumerate}
\end{enumerate}
\end{lem}
\begin{proof}
\par 
We show this lemma by induction on the numbers of Farey sum operation $\sharp $ 
since any irreducible fraction in $(0,1)$ can be obtained from $\frac{0}{1}$ and $\frac{1}{1}$ by applying $\sharp $, repeatedly. 
\par 
The rational number $\frac{1}{2}=[0,1,1]$ is written by $\frac{1}{2}=\frac{0}{1}\sharp \frac{1}{1}$, and thus the lemma holds for $\frac{1}{2}$. 
\par 
Next, suppose that the lemma holds for two rational numbers $\beta , \gamma $. 
Let us show the lemma for $\alpha =\beta \sharp \gamma $.  
We write in the form $\alpha =[0, a_1, \ldots , a_n]$ for some even number $n$. 
\par 
(I) Let us consider the case $a_n\geq 2$. 
Then, 
$$\beta =[0, a_1, \ldots , a_n-1]=\dfrac{x}{r},\quad \gamma =[0, a_1, \ldots , a_{n-1}]=\dfrac{y}{s}.$$
\par 
(i) Suppose that $N_0(a_1, \ldots , a_n)$ is even. Then $N_0(a_1, \ldots , a_n-1)$ is odd. 
\par 
If $\beta $ is $\frac{1}{1}$-type and $\gamma $ is $\frac{0}{1}$-type, then $\alpha $ is $\frac{1}{0}$-type. 
By induction hypothesis, $\beta $ is a Farey sum of rational numbers of $\frac{0}{1}$-type and $\frac{1}{0}$-type. 
Since
$$\beta =\begin{cases}
[0, a_1, \ldots , a_n-2]\sharp [0, a_1, \ldots , a_{n-1}] & \text{if}\ a_n\geq 3,\\ 
[0, a_1, \ldots , a_{n-2}]\sharp [0, a_1, \ldots , a_{n-1}]& \text{if}\ a_n=2, 
\end{cases}
$$
$\gamma =[0, a_1, \ldots , a_{n-1}]$ is $\frac{1}{0}$-type. This is a contradiction. 
By the same manner, we have a contradiction when we suppose that 
$\beta $ and $\gamma $ are $\frac{0}{1}$-type and $\frac{1}{0}$-type, respectively, or that 
$\beta $ and $\gamma $ are $\frac{1}{0}$-type and $\frac{1}{1}$-type, respectively. 
Thus, the following three cases only occur. 
\begin{enumerate}
\item[$\bullet$] $\beta $ and $\gamma $ are $\frac{0}{1}$-type and $\frac{1}{1}$-type, respectively. 
In this case $\alpha $ is $\frac{1}{0}$-type, and $x$ is even. 
\item[$\bullet$] $\beta $ and $\gamma $ are  $\frac{1}{0}$-type and $\frac{0}{1}$-type, respectively. 
In this case $\alpha $ is $\frac{1}{1}$-type, and $y$ is even. 
\item[$\bullet$] 
$\beta $ and $\gamma $ are  $\frac{1}{1}$-type and $\frac{1}{0}$-type, respectively. 
In this case $\alpha $ is $\frac{0}{1}$-type, and $s$ is even. 
\end{enumerate}

(ii) Suppose that $N_0(a_1, \ldots , a_n)$ is odd. Then $N_0(a_1, \ldots , a_n-1)$ is even. 
By the same manner in Part (i), one can verify that the following three cases only occur. 
\begin{enumerate}
\item[$\bullet$] $\beta $ and $\gamma $ are  $\frac{1}{1}$-type and $\frac{0}{1}$-type, respectively. 
In this case $\alpha $ is $\frac{1}{0}$-type, and $y$ is even. 
\item[$\bullet$] $\beta $ and $\gamma $ are  $\frac{0}{1}$-type and $\frac{1}{0}$-type, respectively. 
In this case $\alpha $ is $\frac{1}{1}$-type, and $x$ is even. 
\item[$\bullet$] 
$\beta $ and $\gamma $ are  $\frac{1}{0}$-type and $\frac{1}{1}$-type, respectively. 
In this case $\alpha $ is $\frac{0}{1}$-type, and $r$ is even. 
\end{enumerate}

\par 
(II) Let us consider the case $a_n=1$. Then, 
$$\beta =[0, a_1, \ldots , a_{n-2}]=\dfrac{x}{r},\quad \gamma =[0, a_1, \ldots , a_{n-1}]=\dfrac{y}{s}.$$
\par 
Suppose that $N_0(a_1, \ldots , a_n)$ is even. Then $N_0(a_1, \ldots , a_{n-1}-1,1)$ is odd. 
\par 
If $\beta $ is $\frac{1}{1}$-type and $\gamma $ is $\frac{0}{1}$-type, then $\alpha $ is $\frac{1}{0}$-type. 
By induction hypothesis, $\gamma $ is a Farey sum of rational numbers of $\frac{0}{1}$-type and $\frac{1}{0}$-type. 
Since
$$\gamma =\begin{cases}
[0, a_1, \ldots , a_{n-2}]\sharp [0, a_1, \ldots , a_{n-1}-1] & \text{if}\ a_{n-1}\geq 2,\\ 
[0, a_1, \ldots , a_{n-2}]\sharp [0, a_1, \ldots , a_{n-3}]& \text{if}\ a_{n-1}=1, 
\end{cases}
$$
$\beta =[0, a_1, \ldots , a_{n-2}]$ is $\frac{0}{1}$-type. This is a contradiction. 
By the same argument in Part (I), when $N_0(a_1, \ldots , a_n)$ is even, we see that the same result in (i) of (I) holds. 
It can be also verified that when $N_0(a_1, \ldots , a_n)$ is odd,  the same result in (ii) of (I) holds. 
\end{proof}

\section{Conway-Coxeter friezes of zigzag type}
\par 
A Conway-Coxeter frieze (abbreviated by CCF) \cite{CoCo1, CoCo2, Coxeter} is an infinite array of positive integers, displayed on shifted lines such that the top and bottom lines are composed only of $1$s, and 
each unit diamond 
$$\begin{matrix}
 & b & \\ 
a & & d \\ 
 & c & \end{matrix}$$
in the array satisfies the determinant condition $ad-bc=1$. 
\par 
Given an $LR$ word $w$, one can construct a Conway-Coxeter frieze $\varGamma (w)$. 
We will explain this construction by an example. 
Let $w=L^2R^2L$. 
Then, we set six $1$s as in Figure~\ref{fig1} as an initial arrangement. 

\begin{figure}[htbp]
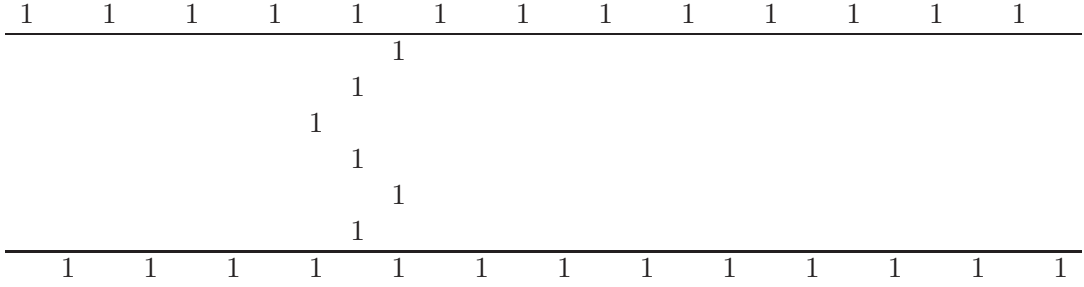

\centering 
$\begin{array}{cccccccccccccccccccccccccc}
1 & & 1 & &1  & & 1 & & 1 & & 1 & & 1 & & 1 & &1  & & 1 &&1&&1&&1& \\\hline 
&  & &  & &  & &  & & 1 & &  & &  & &  & &  & &  &&&&&& \\ 
&&  & &  & &  & & 1 & &   & &  & &  & &  & &  & &  &&&&& \\ 
&&&  & &  & & 1 & &  & &  & &  & &  & &  & &  & &  &&&& \\ 
&&&&  & &  & & 1 & &  & &  & &  & &  & &  & &  & &  &&& \\ 
&&&&&  & &  & & 1 & &  & &  & &  & &  & &  & &  & &  && \\ 
&&&&&&  & & 1 & &  & &  & &  & &  & &  & &  & &  & &  & \\\hline 
&1&&1&&1&&1 & & 1 & & 1 & & 1 & & 1 & & 1 & & 1 & & 1 & & 1 & & 1 \\
\end{array}$
\caption{an initial arrangement of $1$s}\label{fig1}
\end{figure}

The letters \lq\lq $L$" and \lq\lq $R$" correspond to going down to the left and the right in the zigzag path consisting of six $1$s, respectively. 
Applying the rule $ad-bc=1$ repeatedly, we see that 
this initial arrangement generates a Conway-Coxeter frieze, which is given in Figure~\ref{fig2}. 

\begin{figure}[htbp]
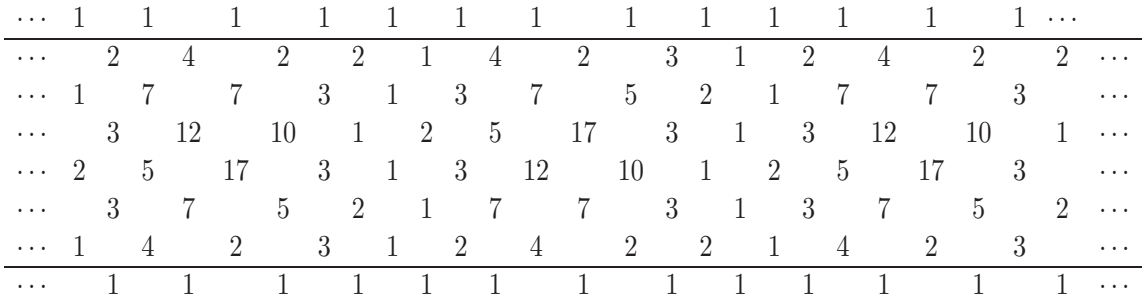

\centering
\scalebox{0.9}[1.05]{
$\begin{array}{c@{\kern0.8em}c@{\kern0.8em}c@{\kern0.8em}c@{\kern0.8em}c@{\kern0.8em}c@{\kern0.8em}c@{\kern0.8em}c@{\kern0.8em}c@{\kern0.8em}c@{\kern0.8em}c@{\kern0.8em}c@{\kern0.8em}c@{\kern0.8em}c@{\kern0.8em}c@{\kern0.8em}c@{\kern0.8em}c@{\kern0.8em}c@{\kern0.8em}c@{\kern0.8em}c@{\kern0.8em}c@{\kern0.8em}c@{\kern0.8em}c@{\kern0.8em}c@{\kern0.8em}c@{\kern0.8em}c@{\kern0.8em}c@{\kern0.8em}c}
\cdots & 1 & & 1 & & 1 & & 1 & & 1 & & 1 & & 1 & & 1 & & 1 & & 1 && 1 && 1 &&1 & \cdots  \\\hline 
\cdots && 2 & & 4 & & 2 & & 2 & & 1 & & 4 & & 2 & & 3 & & 1 & & 2 && 4 && 2 && 2 & \cdots  \\ 
\cdots & 1 && 7 & & 7 & & 3 & & 1 & & 3  & & 7 & & 5 & & 2 & & 1 & & 7 && 7 && 3 && \cdots  \\ 
\cdots && 3 && 12 & & 10 & & 1 & & 2 & & 5 & & 17 & & 3 & & 1 & & 3 & & 12 && 10 && 1 & \cdots  \\ 
\cdots & 2 && 5 && 17 & & 3 & & 1 & & 3 & & 12 & & 10 & & 1 & & 2 & & 5 & & 17 && 3 && \cdots  \\ 
\cdots && 3 &&7 && 5 & & 2 & & 1 & & 7 & & 7 & & 3 & & 1 & & 3 & & 7 & & 5 && 2 & \cdots \\ 
\cdots & 1 && 4 && 2 && 3 & & 1 & & 2 & & 4 & & 2 & & 2 & & 1 & & 4 & & 2 & & 3 && \cdots  \\\hline 
\cdots &&1 &&1 &&1 &&1 & & 1 & & 1 & & 1 & & 1 & & 1 & & 1 & & 1 & & 1 & & 1& \cdots \\
\end{array}$}
\caption{the CCF corresponding to $L^2R^2L$}\label{fig2}
\end{figure}

It is not true that every Conway-Coxeter frieze is constructed from some $LR$ word. 
In fact, there is a Conway-Coxeter frieze such that a $1$-zigzag line connecting the ceiling and the floor does not appear. 
It depends on whether a triangle with only a diagonal line appears in the triangulation of a polygon corresponding to a Conway-Coxeter frieze (see \cite[Remark 3.10]{Kogiso-Wakui}). 
\par 
In the present paper we only consider Conway-Coxeter friezes where $1$-zigzag lines appear. 
Such a CCF is called a Conway-Coxeter frieze of \textit{zigzag type}. 
\par 
Any Conway-Coxeter frieze of zigzag type is constructed from some $LR$ word as explained above. 
This means that the map $w \longmapsto \varGamma (w)$ is a surjection
from the set of $LR$ words to the set of Conway-Coxeter friezes of zigzag type. 
This map is not bijective since $\varGamma ((ir)(w))$ and $\varGamma (w)$ are transformed by a horizontal translation and the reflection with respect to the middle horizontal line each other \cite{Coxeter, Kogiso-Wakui, M-G}. 
We consider that two CCFs are equivalent if they are transformed by the vertical or horizontal reflection or the composition of them. 
\par 
Let $\varGamma $ be a CCF of zigzag type. 
One can find the maximum number, say $q$, in $\varGamma $. 
If $\varGamma =\varGamma (w(\alpha ))$ for some $\alpha \in \mathbb{Q}\cap (0,1)$, then the four numerators 
of $\alpha , i(\alpha ), r(\alpha ), (ir)(\alpha )$ appear around $q$ in $\varGamma $. 
For example, in the case where $\varGamma $ is given by Figure~\ref{fig2}, the maximum number is $17$, and $\alpha =\frac{7}{17},\  
i(\alpha )=\frac{10}{17},\ r(\alpha )=\frac{12}{17},\ (ir)(\alpha )=\frac{5}{17}$, whose numerators appear around $17$ in $\varGamma $. 
In this way, a Conway-Coxeter frieze of zigzag type is determined by the set $\{ w, i(w), r(w), (ir)(w) \}$ for an $LR$ word $w$, or equivalently 
are determined by the set $\{ \alpha , i(\alpha ), r(\alpha ), (ir)(\alpha ) \}$ for a rational number $\alpha $ in the open interval $(0,1)$. 

\section{A characterization of CCFs of zigzag type by rational links}
\par 
Let $\alpha $ be a rational number, and expand it as $\alpha =[a_0, a_1, \ldots , a_n]$ as a continued fraction, where 
$a_0\in \mathbb{Z}$ and $a_1, \ldots , a_n\in \mathbb{N}$. 
Then a $(2, 2)$-tangle diagram $T(\alpha )$ is defined as follows. 
If $n$ is even, then 

\begin{equation}\label{eq3.1}
T(\alpha ) :=\ \raisebox{-1.5cm}{\includegraphics[height=3cm]{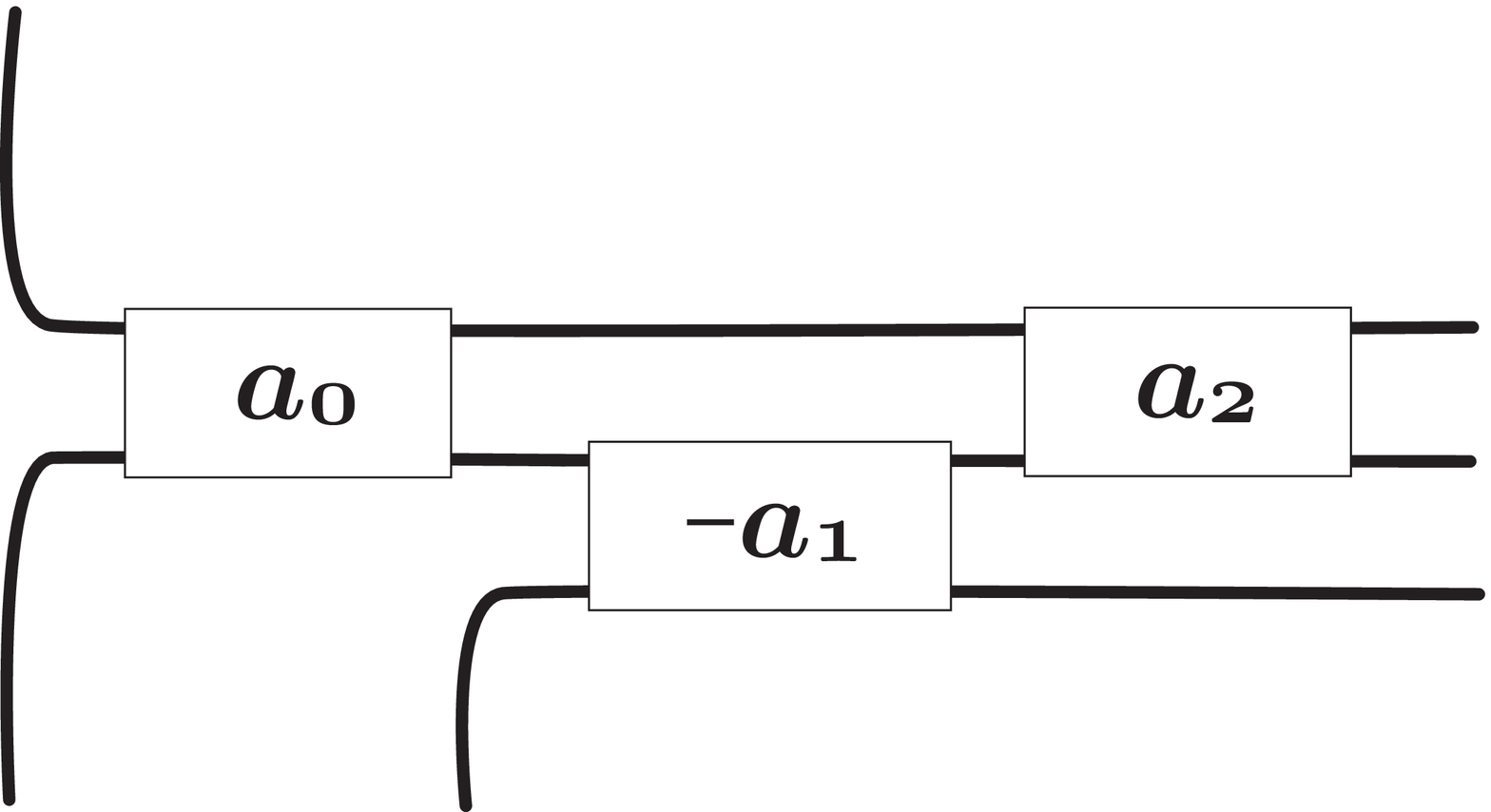}}
\end{equation}

\vspace{0.5cm}
\noindent 
where 
$$\raisebox{-0.3cm}{\includegraphics[height=0.8cm]{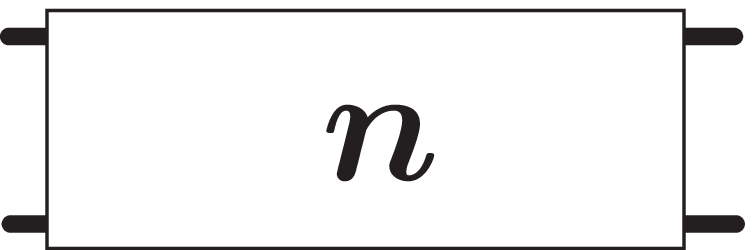}}\ \ =\ 
\begin{cases}
\raisebox{-0.5cm}{\includegraphics[height=1.2cm]{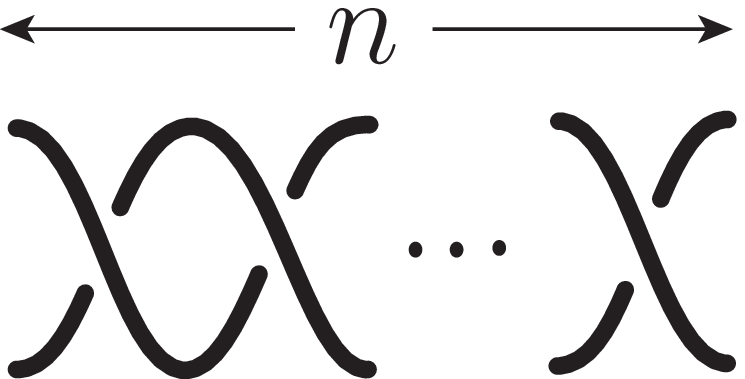}}\ & \raisebox{-0.2cm}{if $n\geq 0$,}\\[0.5cm]  
\raisebox{-0.5cm}{\includegraphics[height=1.2cm]{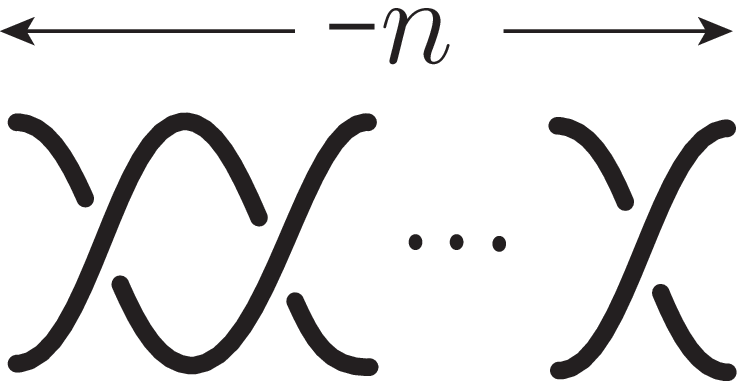}}\ & \raisebox{-0.2cm}{if $n< 0$.}
\end{cases} $$
If $n$ is odd, then 

\begin{equation}\label{eq3.2}
T(\alpha ) :=\ \raisebox{-1.5cm}{\includegraphics[height=3cm]{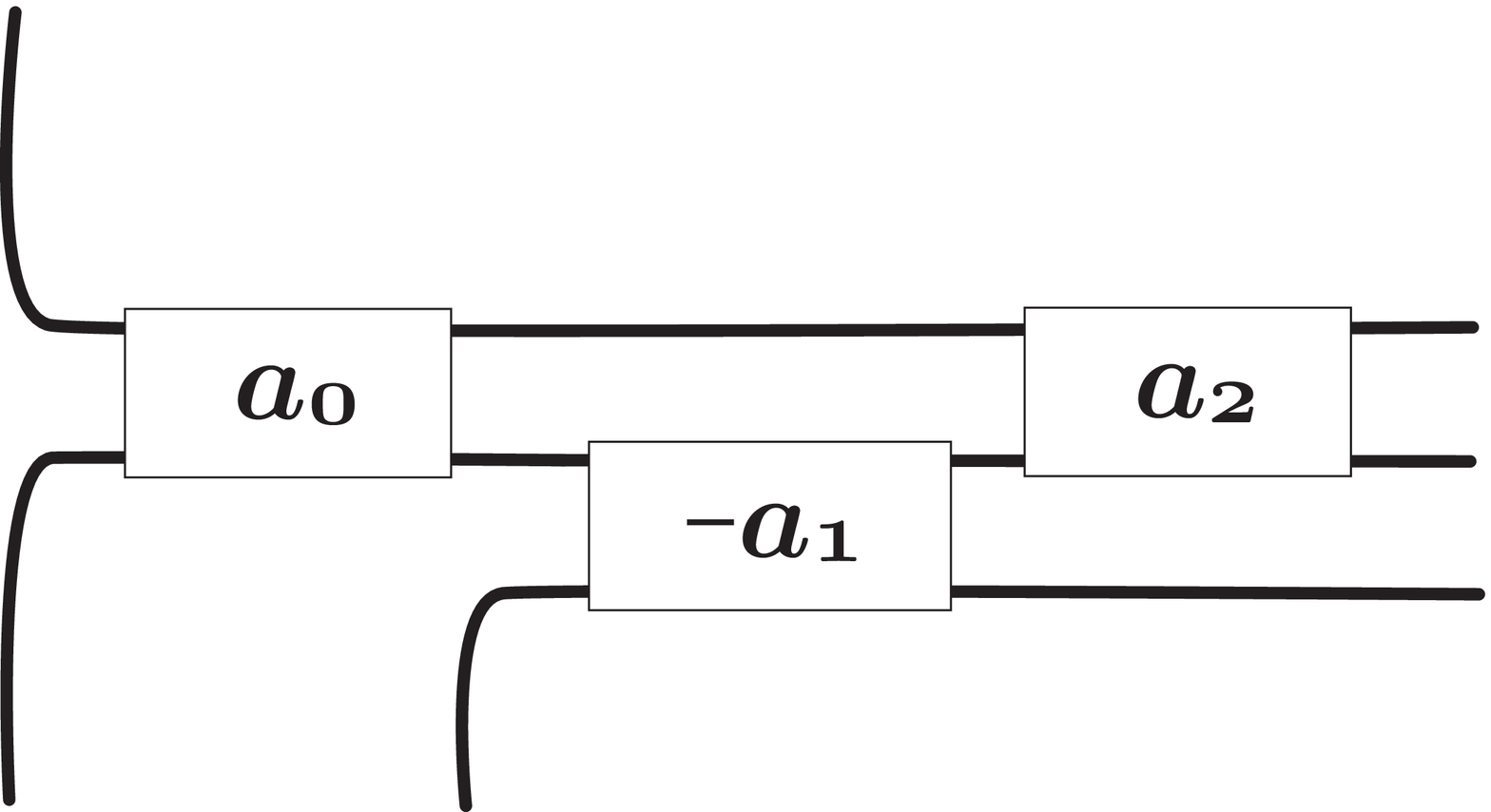}}\ . 
\end{equation}

The definition of $T(\alpha )$ is well-defined since the diagram \eqref{eq3.1} for $a_n\geq 2$ is regular isotopic to the diagram 
$$
\hspace{2cm} \raisebox{-1.5cm}{\includegraphics[height=3cm]{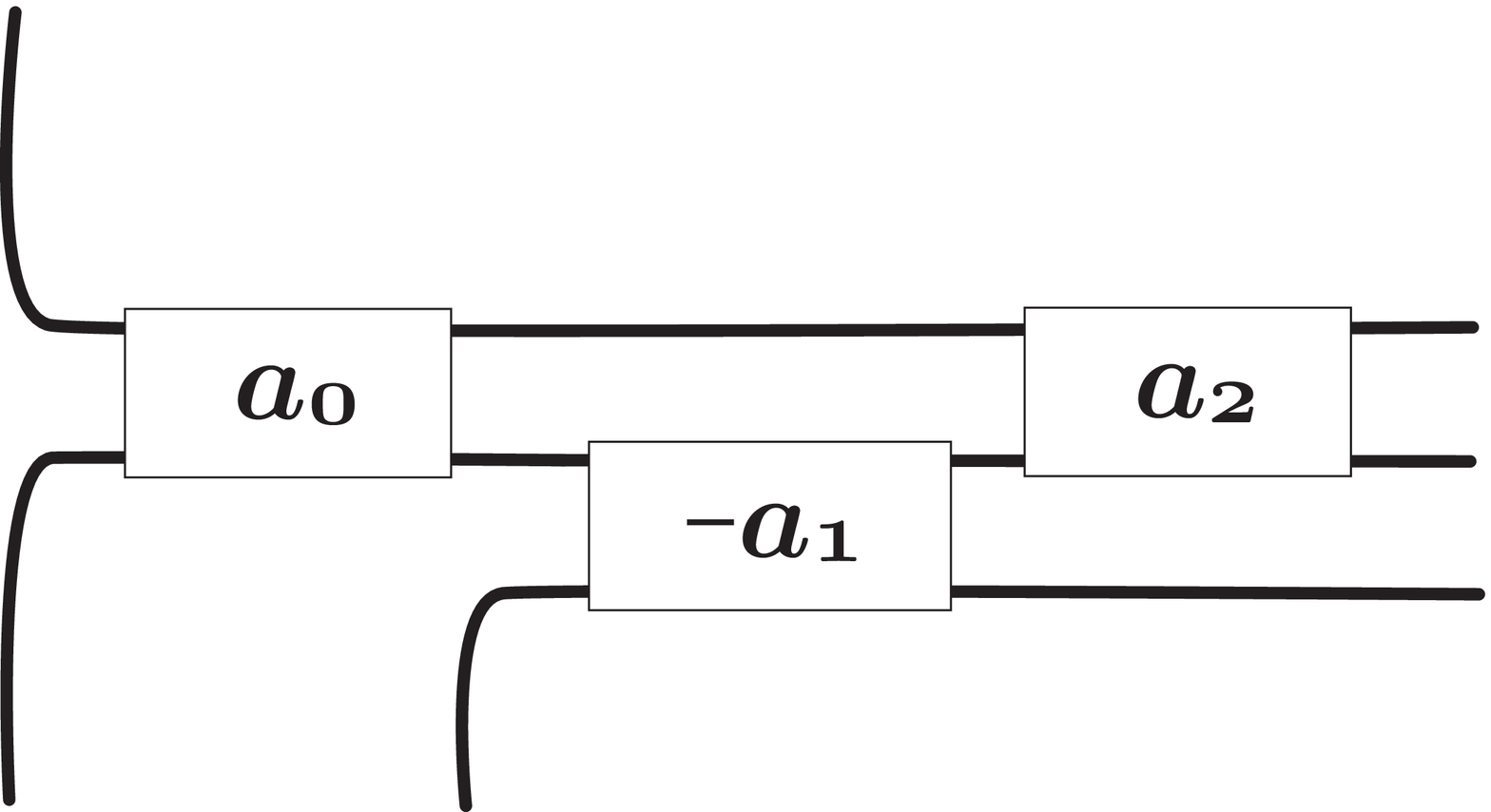}} \ . 
$$

\par \medskip 
For a $(2, 2)$-tangle diagram $T$ the denominator $D(T)$ and the numerator $N(T)$ are links obtained by closing the edge points as Figure~\ref{fig3}. 
We frequently use the same symbols as their diagrams. 

\begin{figure}[hbtp]
\center\includegraphics[height=4cm]{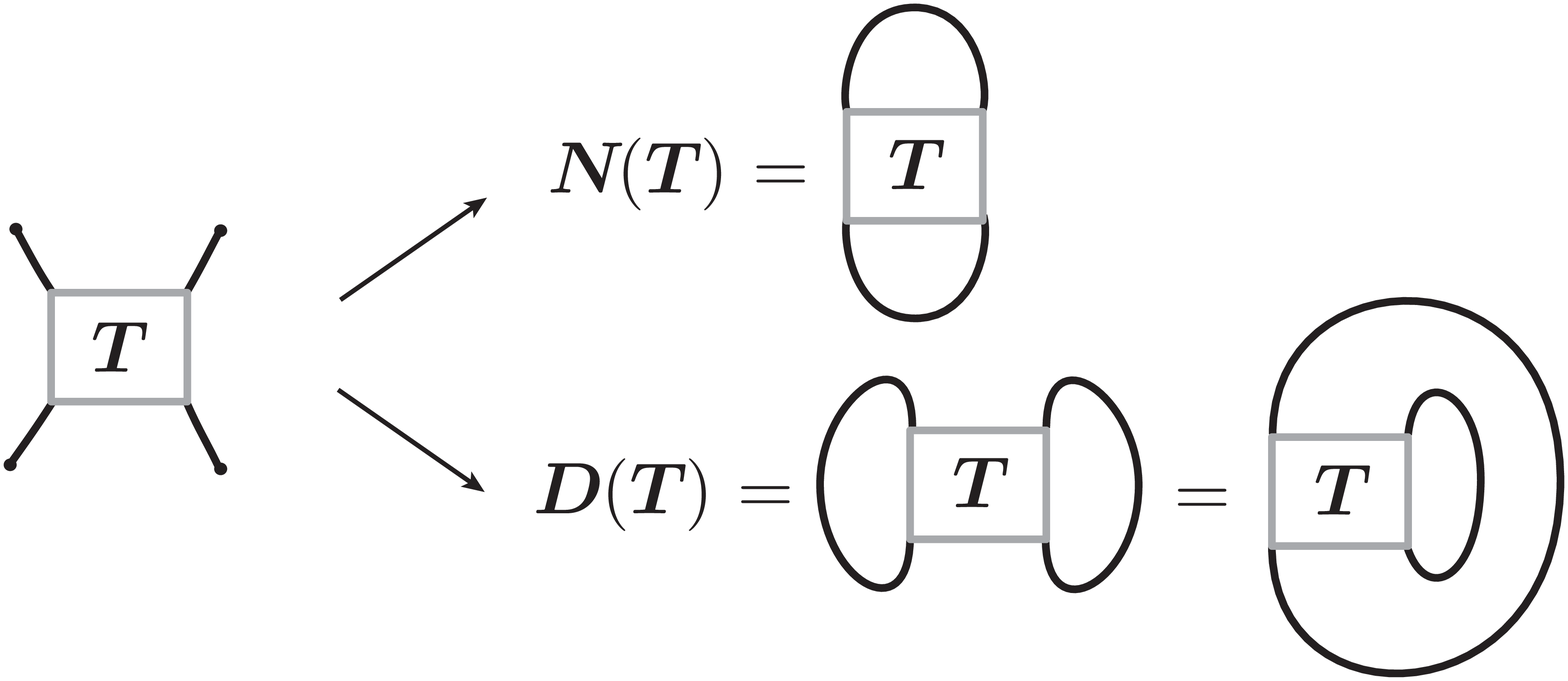}
\caption{the denominator $D(T)$ and the numerator $N(T)$}\label{fig3}
\end{figure}

\par \medskip 
A \textit{rational link} or a \textit{two-bridge link} is a link which is equivalent to the denominator of a rational tangle. 
Such links was classified by Schubert~\cite{Schubert} in 1956. 
The unoriented version of the classification result is as follows. 

\par \medskip 
\begin{thm}[{\bf Schubert}]\label{3-1}
For rational numbers $\alpha =\frac{p}{q}$ and $\beta =\frac{p^{\prime}}{q^{\prime}}$, 
the rational links $D(T(\alpha ))$ and $D(T(\beta ))$ are isotopic as unoriented links if and only if the following two conditions are satisfied: 
\begin{enumerate}
\item[$(1)$] $q=q^{\prime}$, 
\item[$(2)$] $pp^{\prime}\equiv 1\ (\text{mod}\ q)$ or $p\equiv p^{\prime}\ (\text{mod}\ q)$. 
\end{enumerate}
\end{thm} 

\par \medskip 
For a $(2, 2)$-tangle diagram $T$, we denote by $\overline{T}$ the mirror image of $T$, that is obtained by changing over and under at all crossings. 
Theorem~\ref{3-1} can be reformulated in terms of the operations $i, r, ir$ as follows. 

\par \medskip \noindent 
\begin{thm}\label{3-2}
Let $\alpha =\frac{p}{q}\ (q\geq 2)$ be a rational number in $(0,1)$, and $p^{\prime}\in \{ 1, \ldots , q-1\}$. 
We set $\beta =\frac{p^{\prime}}{q}$. 
Then the rational link $D(T(\beta ))$ is isotopic to $D(T(\alpha ))$ or  $D\bigl(\overline{T(\alpha )}\bigr)$ as unoriented links if and only if 
$\beta $ coincides with one of $\alpha , i(\alpha ), r(\alpha ), (ir)(\alpha )$. 
\end{thm}
\begin{proof} 
Let $( \frac{x}{r}, \frac{y}{s})$ be the pair of parents of $\alpha $. 
\par 
$\bullet$ If $\beta =i(\alpha )$, then $\beta =\frac{q-p}{q}$ and $q-p\equiv -p\ (\text{mod}\ q)$. 
By Theorem~\ref{3-1},  
$D(T(i(\alpha )))$ is isotopic to $D\bigl( T\bigl(\frac{-p}{q}\bigr)\bigr) =D(T(-\alpha ))=D\bigl(\overline{T(\alpha )}\bigr)$ as unoriented links. 
\par 
$\bullet$ If $\beta =(ir)(\alpha )$, then $\beta =\frac{s}{q}$. 
Since
$$pr=(x+y)s=ry-1+ys=-1+(r+s)y=-1+qy \equiv -1\ \ (\text{mod}\ q),$$
it follows from Theorem~\ref{3-1} that $D\bigl( T\bigl( (ir)(\alpha )\bigr)\bigr)$ is isotopic to $D\bigl( T\bigl(\frac{-p}{q}\bigr)\bigr) =D(T(-\alpha ))=D\bigl(\overline{T(\alpha )}\bigr)$ as unoriented links. 
\par 
$\bullet$ If $\beta =r(\alpha )$, then $\beta =\frac{s}{q}$. 
Since 
$$pr=(x+y)r=xr+1+xs=x(r+s)+1=qx+1 \equiv 1\ \ (\text{mod}\ q)$$
it follows from Theorem~\ref{3-1} that $D\bigl( T\bigl( r(\alpha )\bigr)\bigr)$ is isotopic to $D(T(\alpha ))$ as unoriented links. 
\par 
Conversely, assume that $p^{\prime}\in \{ 1, \ldots , q-1\}$ satisfies $pp^{\prime}\equiv \pm 1\ (\text{mod}\ q)$ or $p\equiv \pm p^{\prime} \ (\text{mod}\ q)$. 
\par 
$\bullet$ If $pp^{\prime}\equiv 1 \ (\text{mod}\ q)$, then $pp^{\prime}=xq+1$ for some $x\in \mathbb{Z}$. 
Since $pp^{\prime}>0$ and $p^{\prime}<q$, it follows that $p>x\geq 0$. 
Furthermore, 
$\frac{p}{q}=\frac{x}{p^{\prime}}\sharp \frac{p-x}{q-p^{\prime}}$
and 
$p^{\prime}(p-x)-x(q-p^{\prime})=p^{\prime}p-xq=1$. 
Thus $( \frac{x}{p^{\prime}}, \frac{p-x}{q-p^{\prime}})$ is the pair of parents of $\alpha $, and 
$r(\alpha )=\frac{q-p^{\prime}}{q},\ (ir)(\alpha )=\frac{p^{\prime}}{q}$. 
\par 
$\bullet$ If $pp^{\prime}\equiv -1\ (\text{mod}\ q)$, then  $pp^{\prime}=yq-1$ for some $y\in \mathbb{Z}$. 
Since $pp^{\prime}>0$ and $p^{\prime}<q$, it follows that $p>y\geq 0$. 
Furthermore, 
$\frac{p}{q}=\frac{p-y}{q-p^{\prime}}\sharp \frac{y}{p^{\prime}}$
and 
$(q-p^{\prime})y-(p-y)p^{\prime}=qy-p^{\prime}p=1$. 
Thus $( \frac{p-y}{q-p^{\prime}}, \frac{y}{p^{\prime}})$ is the pair of parents of $\alpha $, and 
$r(\alpha )=\frac{p^{\prime}}{q},\ (ir)(\alpha )=\frac{q-p^{\prime}}{q}$. 
\par 
$\bullet$ If $p\equiv p^{\prime}\ (\text{mod}\ q)$, then $p=p^{\prime}$ since $p, p^{\prime}\in \{ 1, \ldots , q-1\}$. 
Therefore, $\alpha =\frac{p^{\prime}}{q}$. 
\par 
$\bullet$ If $p\equiv -p^{\prime}\ (\text{mod}\ q)$, then 
$q-p^{\prime}\equiv p \ (\text{mod}\ q)$. 
Since $p^{\prime}\in \{ 1, \ldots , q-1\}$, it follows that $q-p^{\prime}=p$. 
Thus $i(\alpha )=\frac{q-p}{q}=\frac{p^{\prime}}{q}$. 
\end{proof} 

\par \medskip 
As a corollary of Theorem~\ref{3-2} we have: 

\begin{thm}\label{3-3}
There is a one-to-one correspondence between the Conway-Coxeter friezes of zigzag-type and the sets of pairs of unoriented rational links $\{ D(T(\alpha )), D\bigl(\overline{T(\alpha )}\bigr)\} $. 
\end{thm}

\begin{rem}\label{3-4}
For a $(2, 2)$-tangle diagram $T$, denote by $T^{\textrm{in}}$ the new tangle diagram which is obtained by turning $\overline{T}$ to $90$ degree. 
If $T=T(\alpha )$ for an $\alpha =[0, a_1, \ldots , a_n]$, where $n$ is even, then 
we see that as unoriented links 
\begin{align}
D\bigl( T\bigl( i( \alpha )\bigr)\bigr)
&\sim D\bigl( (T^{\mathrm{in}}\bowtie [-1])\ast [1]\bigr) 
\sim\ \raisebox{-1cm}{\includegraphics[height=2.7cm]{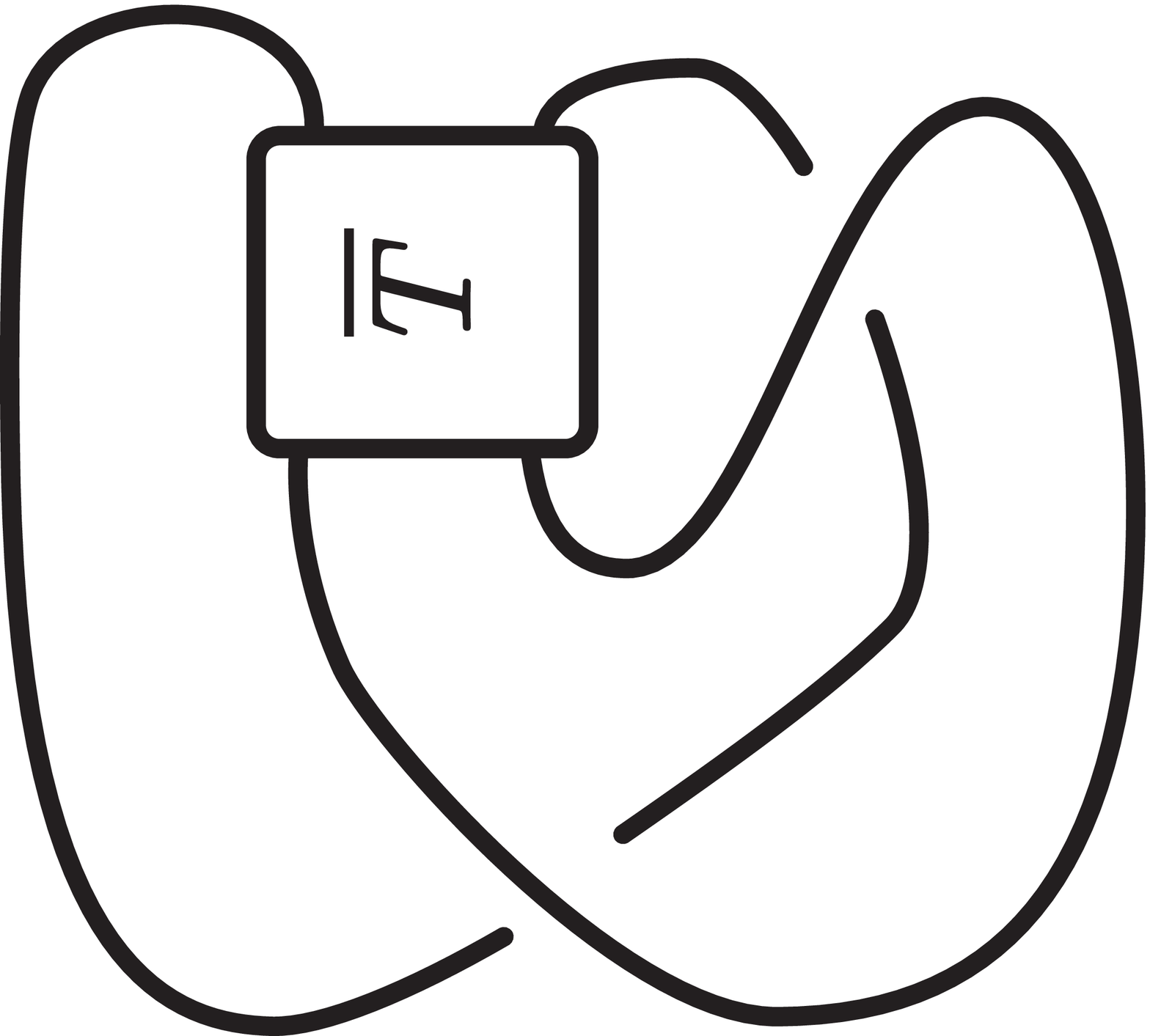}} \label{eq_Rem1}\\[0.3cm]  
&\sim\ \raisebox{-0.6cm}{\includegraphics[height=1.5cm]{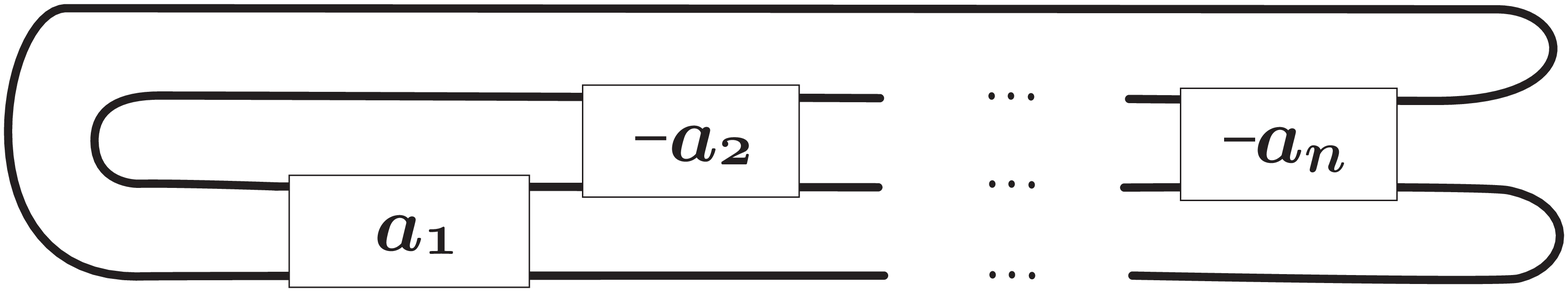}} 
\sim  D\bigl( \overline{T(\alpha )}\bigr), \notag 
\end{align} 
where $\sim$ means regular isotopic on the $2$-dimensional sphere $\mathbb{S}^2$, 
and $\bowtie $ and $\ast$ mean the sum and the product operations on tangle diagrams, respectively. 
The equivalence \eqref{eq_Rem1} follows from Lemma~\ref{1-4}(1) and the Conway's classification result on rational tangles \cite{Conway}.   
\par 
If we set $T(\alpha )^{\mathrm{pal}}=T([0, a_n, \ldots , a_1])$, called the palindrome of $T(\alpha )$, then 
\begin{align}
D\bigl( T\bigl( (ir)( \alpha )\bigr)\bigr) 
&= D\bigl( T(\alpha )^{\mathrm{pal}}\bigr) \label{eq_Rem2}\\[0.3cm]  
&= \ \raisebox{-0.6cm}{\includegraphics[height=1.5cm]{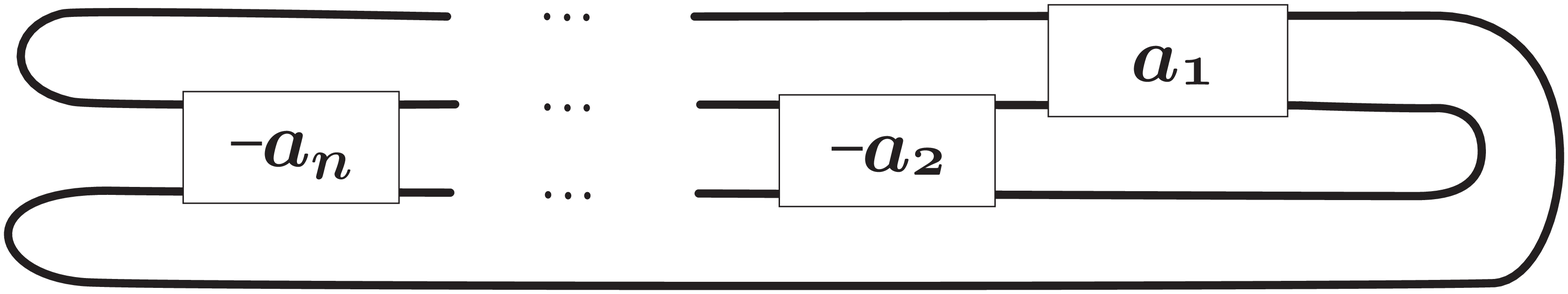}} 
\ =\ D\bigl(\overline{T(\alpha )}\bigr) . \notag 
\end{align} 

See the proof of \cite[Theorem 3.12]{Kogiso-Wakui} for \eqref{eq_Rem2}.  
\par 
From the above observation we also conclude that $D\bigl( T\bigl( r( \alpha )\bigr)\bigr) \sim D\bigl( T(\alpha )\bigr)$. 
\end{rem}

\section{Jones polynomials for the CCFs of zigzag type}
\par 
As an application of Theorem~\ref{3-3}, in this section, we show that ``Jones polynomial" for the CCFs of zigzag type can be defined. 
To describe the formulation we will need to introduce a convention of orientation for the rational links. 
Our convention is the completely same in \cite{LeeSchiffler}. 
\par 
Let $\alpha =\frac{p}{q}$ be an irreducible fraction. 
Note that if $q$ is odd, then the denominator $D( T(\alpha )) $ is a knot, and otherwise it is a two-component link.  
Moreover, we write $\alpha =[0, a_1, \ldots , a_n]$ for some $a_1, \ldots , a_n\in \mathbb{N}$, and 
choose an orientation for $D( T(\alpha )) $ as follows. 
\begin{list}{}{\labelsep=0.1cm \labelwidth=0.8cm \leftmargin=0.8cm}
\item[$\bullet$] If $q$ is odd and $n$ is even, then 
\begin{equation}\label{eq4.1}
D( T(\alpha )) =\ \raisebox{-0.6cm}{\includegraphics[height=1.5cm]{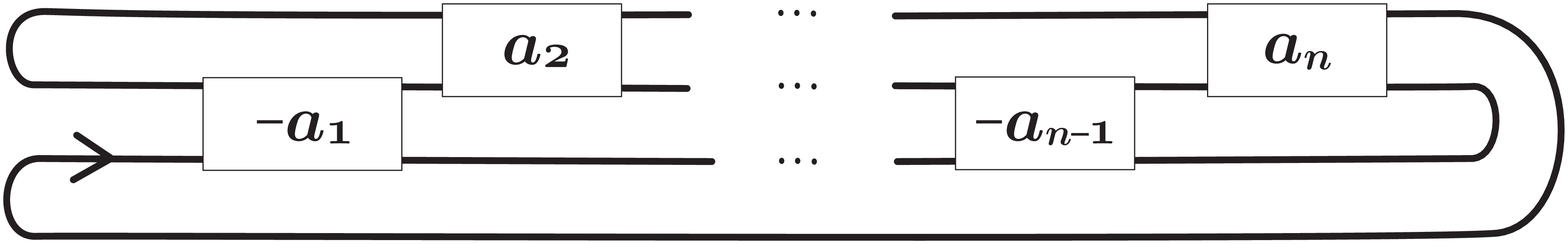}}\ .
\end{equation}
\item[$\bullet$] If $q$ and $n$ are even, then 
\begin{equation}\label{eq4.2}
D( T(\alpha )) =\ \raisebox{-0.6cm}{\includegraphics[height=1.5cm]{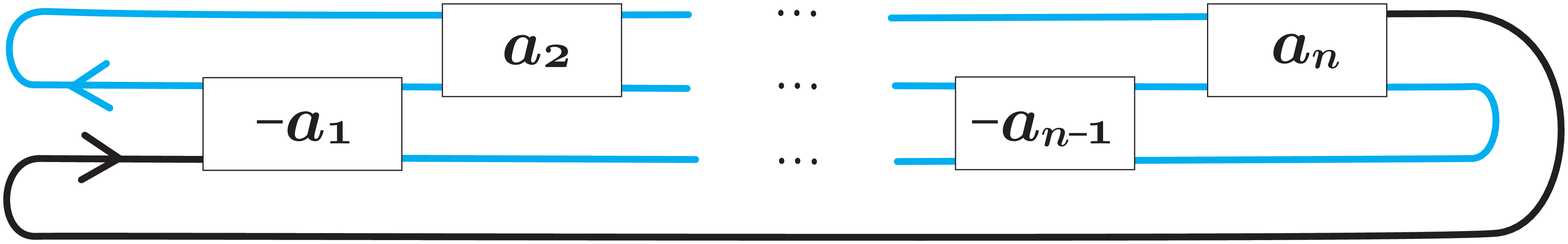}}\ .
\end{equation}
\item[$\bullet$] If $q$ and $n$ are odd, then 
\begin{equation}\label{eq4.3}
D( T(\alpha )) =\ \raisebox{-0.6cm}{\includegraphics[height=1.5cm]{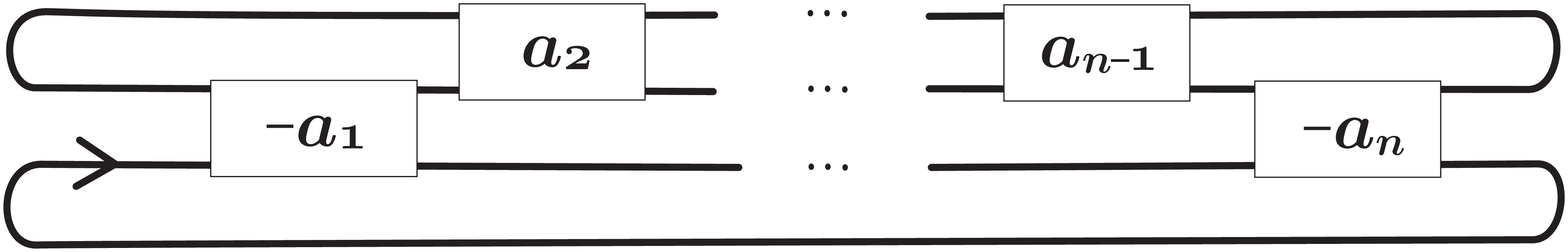}}\ .
\end{equation}
\item[$\bullet$] If $q$ is even and $n$ is odd, then 
\begin{equation}\label{eq4.4}
D( T(\alpha )) =\ \raisebox{-0.6cm}{\includegraphics[height=1.5cm]{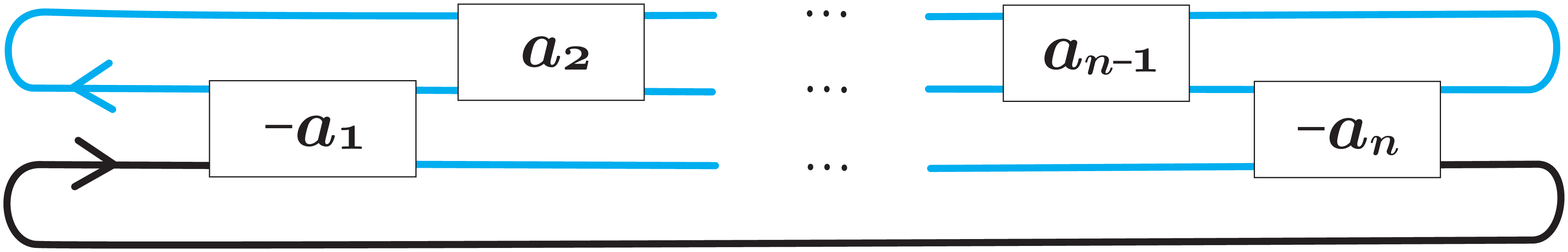}}\ .
\end{equation}
\end{list}

We denote by $\text{wr}(\alpha )$ the writhe of $D\bigl( T(\alpha )\bigr)$ with the above orientation. 

\par 
If $D( T(\alpha )) $ is a two-component link, namely $\alpha $ is $\frac{1}{0}$-type, then we have another orientation for $D( T(\alpha )) $ by changing the given orientation.  
We denote by $D_{+-}(T(\alpha )),\ D_{--}(T(\alpha ))$,\ $D_{-+}(T(\alpha ))$ the obtained links with the following new orientations, respectively. If $n$ is chosen as even, then

\begin{align}
D_{+-}( T( \alpha )) &=\ \raisebox{-0.6cm}{\includegraphics[height=1.5cm]{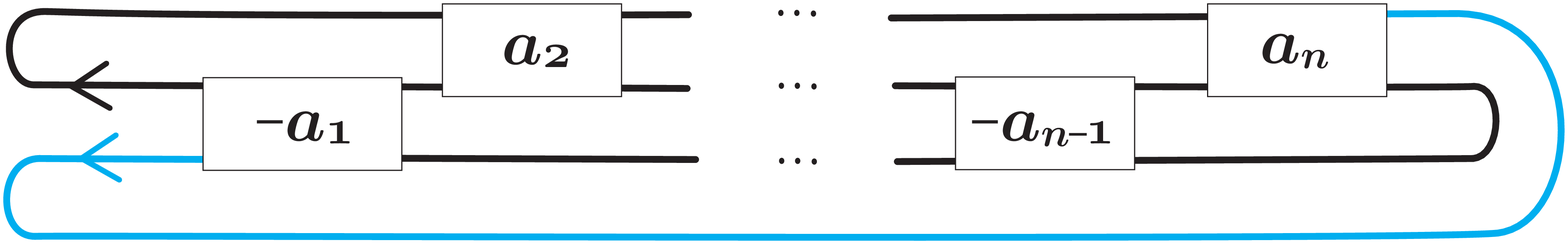}}\ , \\[0.5cm]  
D_{--}( T( \alpha )) &=\ \raisebox{-0.6cm}{\includegraphics[height=1.5cm]{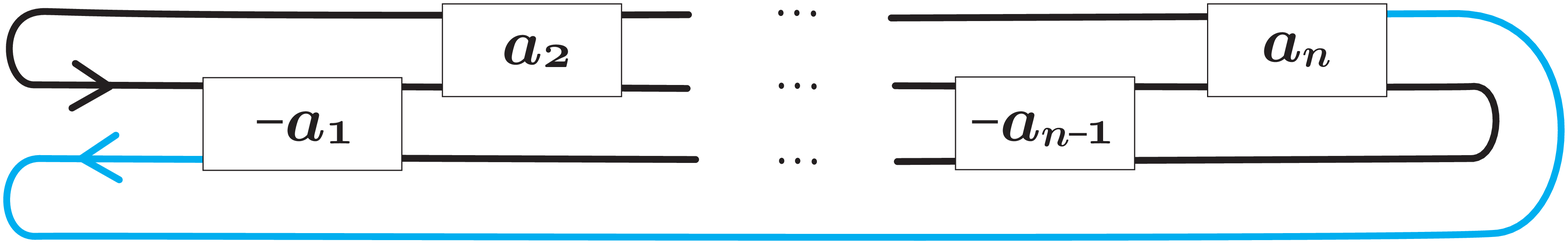}}\ ,  \\[0.5cm]  
D_{-+}( T( \alpha )) &=\ \raisebox{-0.6cm}{\includegraphics[height=1.5cm]{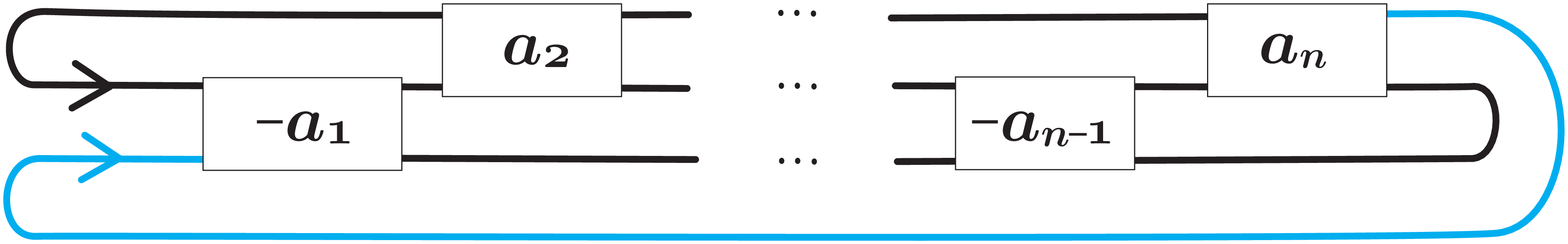}}\ . 
\end{align}

\par \medskip 
In the case where $n$ is chosen as odd, they are defined by similar oriented diagrams. 
\par 
We denote by $\text{wr}_{+-}(\alpha ), \text{wr}_{--}(\alpha ), \text{wr}_{-+}(\alpha )$ the writhes of $D_{+-}(T(\alpha )),\ D_{--}(T(\alpha )),\ D_{-+}(T(\alpha ))$, respectively. 

\par \smallskip 
\begin{lem}\label{4-1}
Let  $\alpha =[0, a_1, a_2, \ldots , a_n]$ be a rational number in $(0,1)$ of type $\frac{1}{0}$. 
\begin{enumerate}
\item[$(1)$] Assume that $n$ is even. 
If $N_0(a_1, \ldots , a_n)$ is even, then 
the oriented diagram $D_{+-}\bigl( T(\alpha )\bigr)$ is regular isotopic to 
$D\bigl( \overline{T\bigl( (ir)(\alpha )\bigr)}\bigr)$. 
If $N_0(a_1, \ldots , a_n)$ is odd, then 
$D_{+-}\bigl( T(\alpha )\bigr)$ is regular isotopic to $D_{-+}\bigl( \overline{T\bigl( (ir)(\alpha )\bigr)}\bigr)$ with orientation. 
Therefore, 
$$\mathrm{wr}_{+-}( \alpha ) 
=\begin{cases}
-\mathrm{wr}\bigl( (ir)(\alpha )\bigr)  & \text{if $N_0(a_1, \ldots , a_n)$ is even},\\ 
-\mathrm{wr}_{+-}\bigl( (ir)(\alpha )\bigr)  & \text{if $N_0(a_1, \ldots , a_n)$ is odd.}
\end{cases}$$
\item[$(2)$] Assume that $n$ is even. 
If $N_0(a_1, \ldots , a_n)$ is even, then 
the oriented diagram $D_{--}\bigl( T(\alpha )\bigr)$ is regular isotopic to 
$D_{-+}\bigl( \overline{T\bigl( (ir)(\alpha )\bigr)}\bigr)$. 
If $N_0(a_1, \ldots , a_n)$ is odd, then $D_{--}\bigl( T(\alpha )\bigr)$ is regular isotopic to 
$D\bigl( \overline{T\bigl( (ir)(\alpha )\bigr)}\bigr)$ with orientation. 
Therefore, 
$$\mathrm{wr}(\alpha )
=\begin{cases}
-\mathrm{wr}_{+-}\bigl( (ir)(\alpha )\bigr)  & \text{if $N_0(a_1, \ldots , a_n)$ is even},\\ 
-\mathrm{wr}\bigl( (ir)(\alpha )\bigr) & \text{if $N_0(a_1, \ldots , a_n)$ is odd.}
\end{cases}$$
\item[$(3)$] The oriented diagram $D\bigl( T(\alpha )\bigr)$ is regular isotopic to 
$D_{-+}\bigl( \overline{T\bigl( i(\alpha )\bigr)}\bigr)$ with orientation, and therefore 
$$\mathrm{wr}(\alpha ) =-\mathrm{wr}_{+-}\bigl( i(\alpha )\bigr) .$$
\end{enumerate} 
\end{lem}
\begin{proof}
By Lemma~\ref{1-4} 
$(ir)(\alpha )=[0, a_n, \ldots , a_2, a_1]$. 
\par 
For an even integer $n$, let $B(\alpha )$ be the following $3$-braid: 
\begin{equation}
B(\alpha )  =\ \raisebox{-0.6cm}{\includegraphics[height=1.2cm]{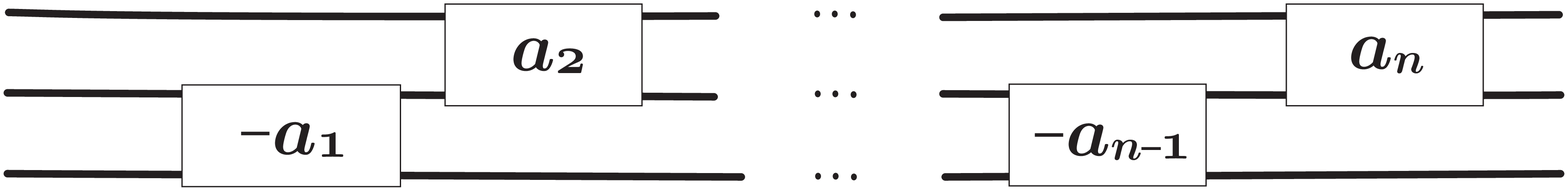}}\ . 
\end{equation}

\vspace{0.2cm} 
Assume that the both three terminal points of $B(\alpha )$ are numbered as $1,2,3$ beginning at the top. 
When the $j$th terminal point on the left is connected with the $i_j$th terminal point on the right by a strand of $B(\alpha )$,  
we set 
$$\sigma (\alpha ) :=\begin{pmatrix}
1 & 2 & 3 \\ 
i_1 & i_2 & i_3 \end{pmatrix}.$$
The diagram $D( T(\alpha ) )$ has two-component if and only if the number $3$ is sent to the number $1$ under this permutation $\sigma (\alpha ) $. 
\par 
(1) If the number $N_0(a_1, \ldots , a_n)$ is even, then $\sigma (\alpha )$ sends $1$ to $2$. 
The oriented diagram $D_{+-}( T(\alpha ))$ can be deformed as 
\par \medskip 
\begin{align*}
D_{+-}\bigl( T(\alpha )\bigr)
&=\ \raisebox{-0.6cm}{\includegraphics[height=1.5cm]{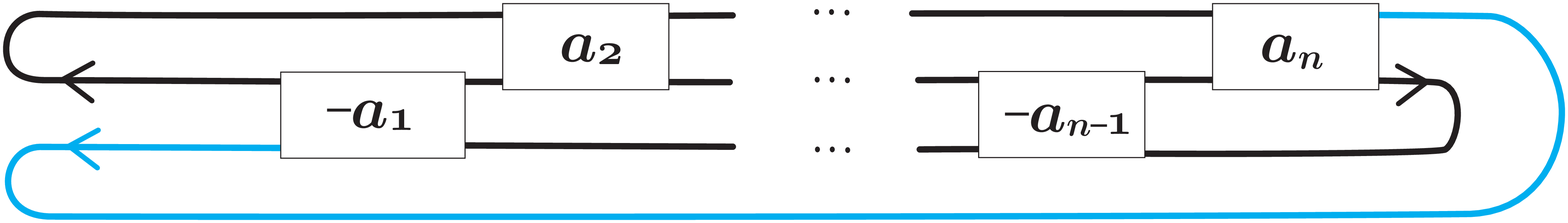}} \\[0.5cm] 
&\sim \ \raisebox{-0.6cm}{\includegraphics[height=1.5cm]{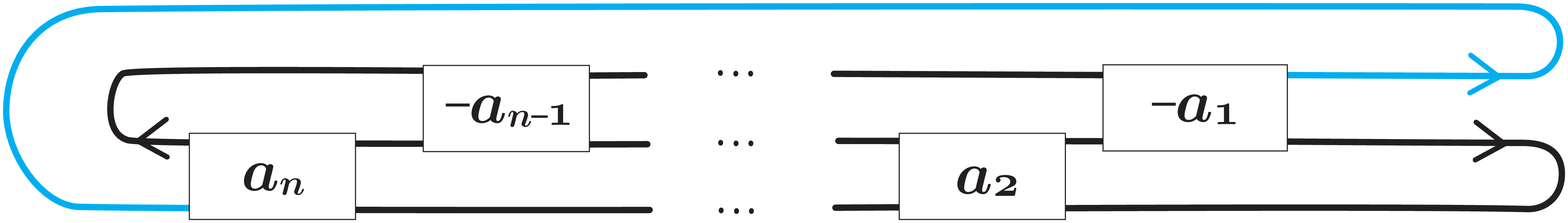}} \\[0.5cm]  
&\sim \ \raisebox{-0.6cm}{\includegraphics[height=1.6cm]{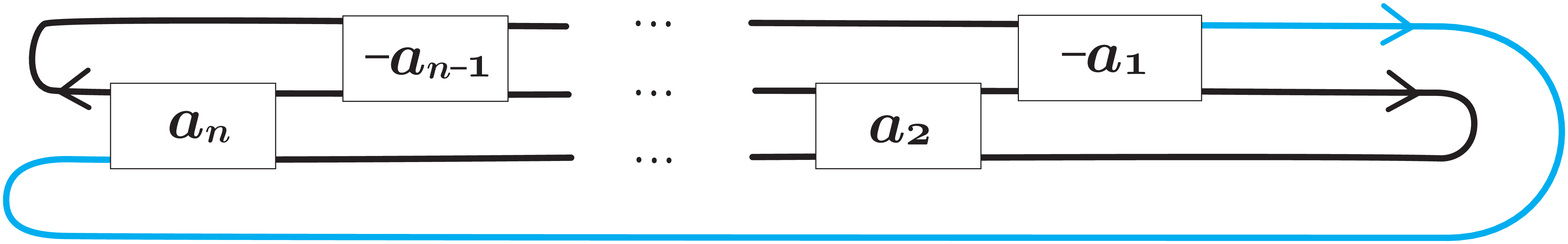}} \\[0.5cm]  
&\sim \ \raisebox{-0.6cm}{\includegraphics[height=1.45cm]{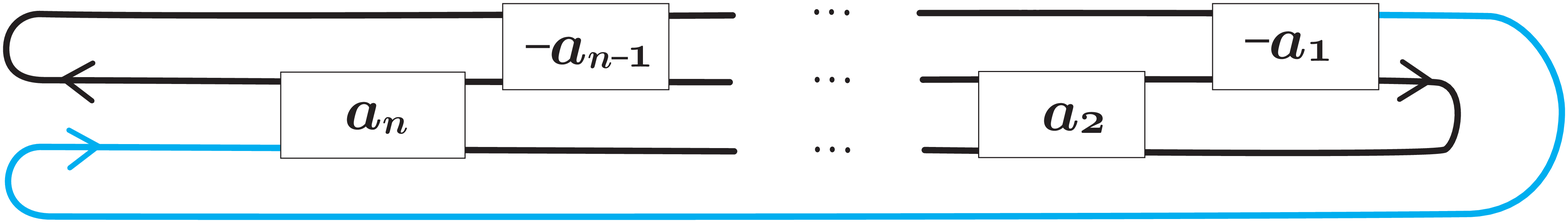}} \\[0.5cm]  
&\sim \ D\bigl(  \overline{T\bigl( (ir)(\alpha )\bigr)}\bigr) . 
\end{align*}
Thus we have 
$$\text{wr}_{+-}(\alpha ) 
=\text{wr}\Bigl( D\bigl( \overline{T\bigl( (ir)(\alpha )\bigr)}\bigr)\Bigr) 
=-\text{wr}\bigl( (ir)(\alpha )\bigr) .$$

If $N_0(a_1, \ldots , a_n)$ is odd, then $\sigma (\alpha )$ sends $1$ to $3$. 
By deforming $D_{+-}(T(\alpha ))$ in the same manner, we have 
\begin{align*}
D_{+-}(T(\alpha ))
&\sim\ 
\raisebox{-0.6cm}{\includegraphics[height=1.5cm]{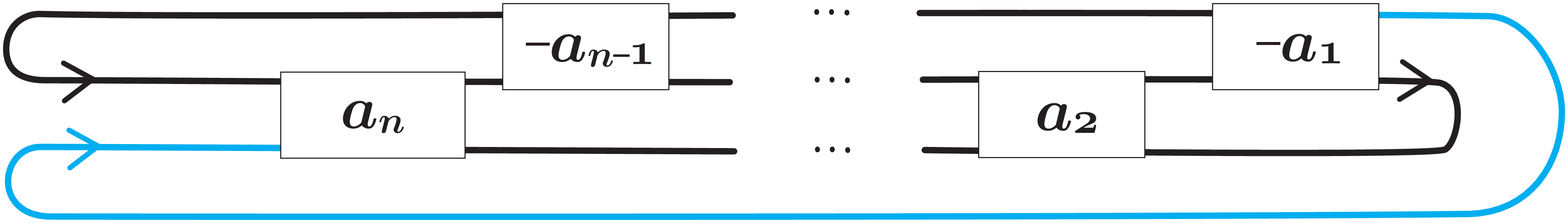}} \\[0.5cm] 
&\sim \ D_{-+}\bigl( \overline{T\bigl( (ir)(\alpha )\bigr)}\bigr) ,  
\end{align*}
and hence 
$$\text{wr}_{+-}(\alpha )
=\text{wr}\Bigl( D_{-+}\bigl( \overline{T\bigl( (ir)(\alpha )\bigr)}\bigr)\Bigr) 
=\text{wr}\Bigl( D_{+-}\bigl( \overline{T\bigl( (ir)(\alpha )\bigr)}\bigr)\Bigr) 
=-\text{wr}_{+-}\bigl( (ir)(\alpha )\bigr) . 
$$
\indent 
(2) By the same method if $N_0(a_1, \ldots , a_n)$ is even, then $D_{--}(T(\alpha ))\sim \ D_{-+}\bigl( \overline{T\bigl( (ir)(\alpha )\bigr)}\bigr)$, and hence 
$$\text{wr}(\alpha )
=\text{wr}_{--}(\alpha )
=\text{wr}\Bigl( D_{-+}\bigl( \overline{T\bigl( (ir)(\alpha )\bigr)}\bigr)\Bigr) 
=-\text{wr}_{+-}\bigl( (ir)(\alpha )\bigr) .$$
If $N_0(a_1, \ldots , a_n)$ is odd, then $D_{--}(T(\alpha ))\sim \ D\bigl( \overline{T\bigl( (ir)(\alpha )\bigr)}\bigr)$, and hence 
$$\text{wr}(\alpha ) 
=\text{wr}_{--}(\alpha ) 
=-\text{wr}\bigl( (ir)(\alpha )\bigr) .$$
\indent 
(3) Assume that $a_n\geq 2$. 
By Lemma~\ref{1-4}, $i(\alpha )=[0, 1, a_1-1, a_2, \ldots , a_n]$. 
\par 
If $n$ is even, then 
\begin{align*}
D_{-+}(T(i(\alpha )))
&=\ \raisebox{-0.6cm}{\includegraphics[height=1.5cm]{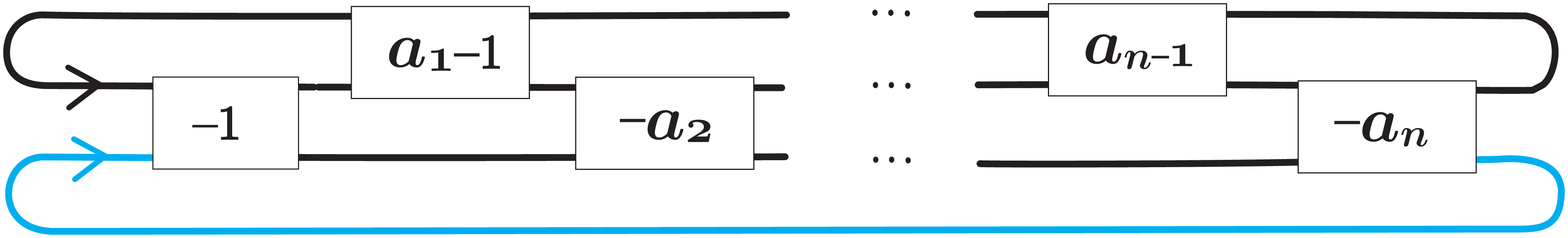}} \\[0.5cm] 
&\sim \ \raisebox{-0.6cm}{\includegraphics[height=1.5cm]{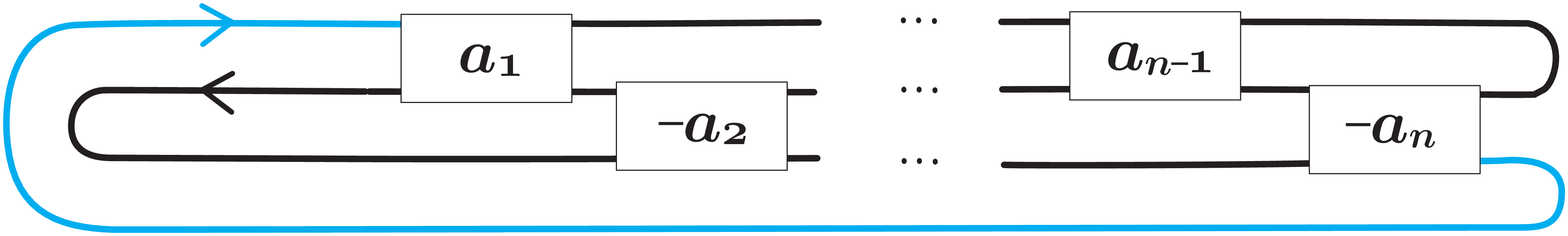}} \\[0.5cm]  
&\sim \ \raisebox{-0.6cm}{\includegraphics[height=1.5cm]{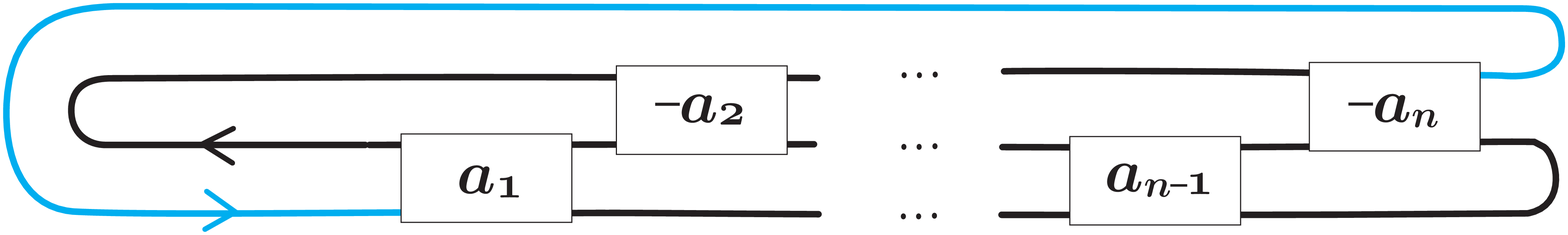}} \\[0.5cm]  
&\sim \ \raisebox{-0.6cm}{\includegraphics[height=1.45cm]{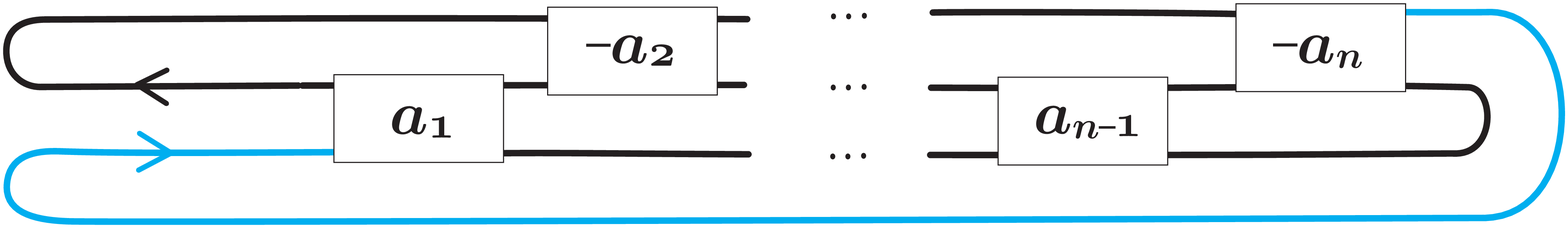}} \\[0.5cm]  
&\sim \ D\bigl( \overline{T(\alpha )}\bigr) .  
\end{align*}
Thus $D_{-+}\bigl(\overline{T(i(\alpha ))}\bigr)\sim \ D\bigl( T(\alpha )\bigr) $, 
and therefore 
$\text{wr}(\alpha ) =-\text{wr}_{-+}\bigl( i(\alpha )\bigr) =-\text{wr}_{+-}\bigl( i(\alpha )\bigr) $. 
\par 
In the case where $n$ is odd, applying the same method we have the same result. 
\end{proof} 

\par \smallskip 
\begin{lem}\label{4-2}
Let  $\alpha =\frac{p}{q}$ be a rational number in $(0,1)$, and assume that $q$ is odd. Then
$$\mathrm{wr}(\alpha )=\mathrm{wr}\bigl( (ir)(\alpha ) \bigr)=-\mathrm{wr}\bigl( i(\alpha )\bigr) .$$
\end{lem}
\begin{proof} 
Since $q$ is odd, $D(T(i(\alpha )))$ is a knot diagram. 
So, $D(T(i(\alpha )))\sim D(\overline{T(\alpha )})$ and 
$\text{wr}(\alpha ) =-\text{wr}\bigl( i(\alpha )\bigr)$. 
\par 
Let $D_{-}(T(\alpha ))$ be the oriented diagram $D(T(\alpha ))$ with the opposite orientation. 
By a similar manner in the proof of Lemma~\ref{4-1}(1), 
$D_{-}(T(\alpha ))$ is regular isotopic to $D\bigl( \overline{T\bigl( (ir)(\alpha )\bigr)}\bigr)$. 
It follows that 
$\text{wr}(\alpha )  =\text{wr}\bigl( D_{-}\bigl(T(\alpha )\bigr)\bigr) 
=-\text{wr}\bigl( (ir)(\alpha ) \bigr)$. 
\end{proof} 

Let $V(\alpha )$ be the Jones polynomial of the oriented link given by the diagram $D\bigl( T(\alpha )\bigr)$ with orientation given by \eqref{eq4.1} -- \eqref{eq4.4}. 
Then 
\begin{equation}\label{eq4.9}
V(\alpha )=(-A^3)^{-\text{wr}(\alpha )}\bigl\langle D\bigl( T(\alpha )\bigr) \bigr\rangle ,  
\end{equation}
where the bracket $\langle \ \ \rangle $ means the Kauffman bracket polynomial \cite{Kau-Topology}, which is a Laurent polynomial in variable $A$ with integer coefficient and is defined by the following axioms. 
\begin{enumerate}\itemindent=1cm 
\item[(KB1)] $\langle \ \raisebox{-0.15cm}{\includegraphics[width=0.5cm]{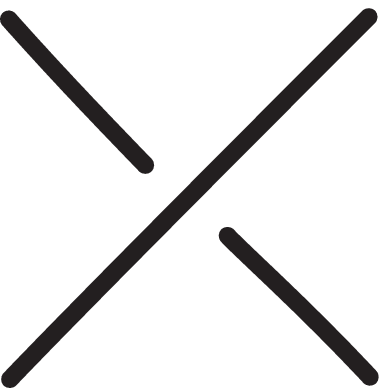}}\ \rangle 
=A\langle \ \raisebox{-0.15cm}{\includegraphics[width=0.5cm]{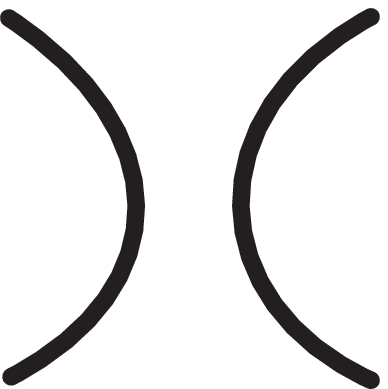}} \ \rangle 
+A^{-1}\langle \ \raisebox{-0.15cm}{\includegraphics[width=0.5cm]{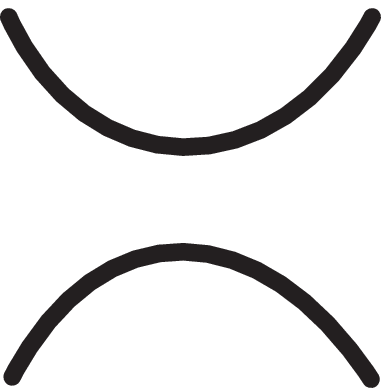}} \ \rangle $ 
\item[(KB2)] $\langle \ D\coprod \raisebox{-0.1cm}{\includegraphics[width=0.4cm]{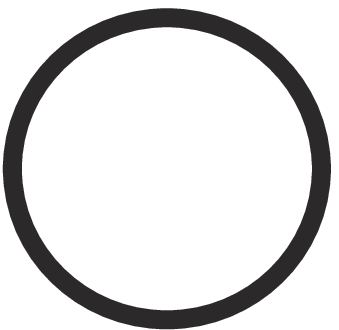}}\ \rangle =\delta \langle  D\rangle $,  where $\delta =-A^2-A^{-2}$.
\item[(KB3)] $\langle \ \raisebox{-0.1cm}{\includegraphics[width=0.4cm]{unoriented_circle.eps}}\ \rangle =1$. 
\item[(KB4)] $\langle D\rangle $ is a regular isotopy invariant of $D$, that is, it is invariant under Reidemeister moves II and III.
\end{enumerate}

\par 
The Kauffman bracket polynomials for Conway-Coxeter friezes of zigzag-type are introduced in \cite{Kogiso-Wakui}. 
To describe it, we use the weight $\text{wt}(\alpha )$ for a positive rational number $\alpha $ which is defined as follows. 
Express $\alpha $ as 
$\alpha =[a_0, a_1, \ldots , a_n]$ such that $a_0$ is a non-negative integer and $a_1, \ldots , a_n$ are positive integers, and set  
\begin{equation}
\text{wt}(\alpha )=\begin{cases}
a_0-a_1+a_2-\cdots +a_n & \text{if $n$ is even}, \\ 
a_0-a_1+a_2-\cdots -a_n+2 & \text{if $n$ is odd}. 
\end{cases}
\end{equation}
\noindent 
It can be easily shown that $\text{wt}(\alpha )$ is well-defined. 
\par 
For the Conway-Coxeter frieze $\varGamma _{\alpha}$ corresponding to a rational number $\alpha $ in $(0,1)$, the Kauffman bracket polynomial $\langle \varGamma _{\alpha }\rangle $ is given by the formula (see \cite[Theorem 2.10, Equation (2.7)]{Kogiso-Wakui}): 
\begin{equation}\label{eq4.11}
\langle \varGamma _{\alpha }\rangle =(-A^3)^{\text{wt}(\alpha )}\bigl\langle D( T( \alpha )) \bigr\rangle .
\end{equation}

Let us introduce the extended weight $\widetilde{\text{wt}}(\alpha )$ by 
\begin{equation}\label{eq4.12}
\widetilde{\text{wt}}(\alpha )=-\text{wr}(\alpha )-\text{wt}(\alpha ) . 
\end{equation}

Combining \eqref{eq4.11} and \eqref{eq4.9} we have: 

\par \bigskip 
\begin{prop}\label{4-3}
For a rational number $\alpha $ in $(0,1)$, 
the Jones polynomial $V(\alpha )$ of the rational link $D(T(\alpha ))$ is given by 
\begin{equation}\label{eq4.13}
V(\alpha )=(-A^3)^{\widetilde{\mathrm{wt}}(\alpha )}\langle \varGamma_{\alpha }\rangle . 
\end{equation}
\end{prop} 

\begin{rem}\label{4-4}
Nagai and Terashima \cite[Theorem 4.4]{Nagai_Terashima} found a combinatorial formula for the writhe $\mathrm{wr}(\alpha )$. 
If we write in the form $\alpha =[0, a_1, \ldots , a_n]$ by some $a_1, \ldots , a_n\in \mathbb{N}$, then the writhe is given by 
\begin{equation}\label{eq4.14}
-\mathrm{wr}(\alpha ) =\sum\limits_{j=1}^nt_{\alpha }(\Delta _j)a_j.
\end{equation}
Here, $\{ \Delta _1,\ldots , \Delta _n\}$ is the sequence of triangles determined by the continued fraction expansion $\alpha =[0, a_1, \ldots , a_n]$, and $t_{\alpha }(\Delta _j)$ is a sign of $\Delta _j$ determined as follows. 
\begin{align*}
t_{\alpha }(\Delta _1) &=\begin{cases}
1 & \text{if $\alpha $ is $\frac{1}{0}$ or $\frac{0}{1}$-type}, \\ 
-1 & \text{if $\alpha $ is $\frac{1}{1}$-type},
\end{cases},\\ 
t_{\alpha }(\Delta _j) &=\begin{cases}
-t_{\alpha }(\Delta _{j-1})  & \text{if the Seifert path of $\alpha $ goes through between $\Delta _{j-1}$ and $\Delta _j$}, \\ 
t_{\alpha }(\Delta _{j-1})  & \text{otherwise}. 
\end{cases}
\end{align*}

More direct recursive formula is given in the last section. 
Thus, the extended weight $\widetilde{\text{wt}}(\alpha )$ is computable in a purely combinatorial way. 
\end{rem} 

If a rational number $\alpha$ is $\frac{1}{0}$-type, then we define $\widetilde{\mathrm{wt}}_{+-}(\alpha )$ by 
\begin{equation}
\widetilde{\mathrm{wt}}_{+-}(\alpha ):=-\text{wr}_{+-}(\alpha )-\text{wt}(\alpha ) . 
\end{equation}

\par \bigskip \noindent 
\begin{lem}\label{4-5}
Let $p, q$ be coprime integers satisfying with $0<p<q$, and $( \frac{x}{s}, \frac{y}{r})$ be the pair of parents of 
$\alpha =\frac{p}{q}$. 
\begin{enumerate}
\item[$(1)$] If $q$ is odd, then 
$$\widetilde{\mathrm{wt}}\bigl( i(\alpha )\bigr) 
=\widetilde{\mathrm{wt}}\bigl( (ir)(\alpha )\bigr) 
=-\widetilde{\mathrm{wt}}(\alpha ),\quad 
\widetilde{\mathrm{wt}}\bigl( r(\alpha )\bigr) 
=\widetilde{\mathrm{wt}}(\alpha ).$$
\item[$(2)$] If $q$ and $x$ are even, then 
$$\widetilde{\mathrm{wt}}\bigl( i(\alpha )\bigr) 
=\widetilde{\mathrm{wt}}\bigl( (ir)(\alpha )\bigr) 
=-\widetilde{\mathrm{wt}}_{+-}(\alpha ),\quad 
\widetilde{\mathrm{wt}}\bigl( r(\alpha )\bigr) 
=\widetilde{\mathrm{wt}}(\alpha ).$$
\item[$(3)$] If $q$ and $y$ are even, then
$$\widetilde{\mathrm{wt}}\bigl( i(\alpha )\bigr) 
=-\widetilde{\mathrm{wt}}_{+-}(\alpha ),\quad 
\widetilde{\mathrm{wt}}\bigl( (ir)(\alpha )\bigr) 
=-\widetilde{\mathrm{wt}}(\alpha ),\qquad 
\widetilde{\mathrm{wt}}\bigl( r(\alpha )\bigr) 
=\widetilde{\mathrm{wt}}_{+-}(\alpha ).$$
\end{enumerate}
\end{lem}
\begin{proof} 
Let $\alpha $ express as a continued fraction $\alpha =[0, a_1, \ldots , a_n]$, where $n$ is even.  
\par 
(1) If $q$ is odd, then by Lemma~\ref{4-2}, $\mathrm{wr}\bigl( i(\alpha )\bigr)=-\mathrm{wr}(\alpha )$. 
Since $n$ is even, $i(\alpha )=[0, 1, a_1-1, a_2, \ldots , a_n]$ and 
$$\mathrm{wt}\bigl( i(\alpha )\bigr) 
=-1+(a_1-1)+\sum\limits_{k=2}^n (-1)^k a_k
=-\sum\limits_{k=1}^n (-1)^{k-1}a_k
=-\mathrm{wt}(\alpha ).$$
Thus, we have
$$\widetilde{\mathrm{wt}}\bigl( i(\alpha )\bigr) 
=-\mathrm{wr}\bigl( i(\alpha )\bigr) -\mathrm{wt}\bigl( i(\alpha )\bigr) 
=\mathrm{wr}(\alpha )+\mathrm{wt}(\alpha ) 
=-\widetilde{\mathrm{wt}}(\alpha ).$$
By the same manner, 
since $\mathrm{wr}\bigl( (ir)(\alpha )\bigr)=-\mathrm{wr}(\alpha )$ and 
$\mathrm{wt}\bigl( (ir)(\alpha )\bigr) =-\mathrm{wt}(\alpha )$, it follows that  
$\widetilde{\mathrm{wt}}\bigl( (ir)(\alpha )\bigr) =-\widetilde{\mathrm{wt}}(\alpha )$. 
\par 
If we set $\beta =(ir)(\alpha )$, then $r(\alpha )=i(\beta )$, and hence 
$$\mathrm{wr}\bigl( r(\alpha )\bigr) 
=\mathrm{wr}\bigl( i(\beta )\bigr) 
=-\mathrm{wr}( \beta )
=-\mathrm{wr}\bigl( (ir)(\alpha ) \bigr) 
=\mathrm{wr}( \alpha ).$$ 
Since $n$ is even, $r(\alpha )=[0, 1, a_n-1, a_{n-1}, \ldots , a_1]$ and 
$$\mathrm{wt}\bigl( r(\alpha )\bigr) 
=-1+(a_n-1)+\sum\limits_{k=2}^n (-1)^{k-1}a_{n-k+1}+2
=a_n-\sum\limits_{k=1}^{n-2} (-1)^{n-k+1}a_k
=\sum\limits_{k=1}^n (-1)^k a_k
=\mathrm{wt}(\alpha ).$$
Thus, we have $\widetilde{\mathrm{wt}}\bigl( r(\alpha )\bigr) =\widetilde{\mathrm{wt}}( \alpha )$. 
\par 
(2) Since $q$ is even, 
$\mathrm{wr}\bigl( i(\alpha )\bigr) =-\mathrm{wr}_{+-}( \alpha )$ by Lemma~\ref{4-1}(3). 
Since $n$ is even, as the proof of Part (1) one can show that 
$\mathrm{wt}\bigl( i(\alpha )\bigr) =-\mathrm{wt}(\alpha )$. 
Thus, 
$\widetilde{\mathrm{wt}}\bigl( i(\alpha )\bigr) =-\widetilde{\mathrm{wt}}_{+-}( \alpha )$. 
\par 
Since $n$ is even, as the proof of Part (1) one can show that 
$\mathrm{wt}\bigl( (ir)(\alpha )\bigr) 
=-\mathrm{wt}(\alpha )$. 
\par 
If $x$ is even, then $N_0(a_1, \ldots , a_n)$ is also even by Lemma~\ref{1-5}.  
Thus, $\mathrm{wr}\bigl( (ir)(\alpha )\bigr) =-\mathrm{wr}_{+-}(\alpha )$ by Lemma~\ref{4-1}(1), and 
therefore, 
$$\widetilde{\mathrm{wt}}\bigl( (ir)(\alpha )\bigr)
 =\mathrm{wr}_{+-}(\alpha )+\mathrm{wt}(\alpha )=-\widetilde{\mathrm{wt}}_{+-}( \alpha ).$$
\par 
If we set $\beta =(ir)(\alpha )$, then 
$r(\alpha )=i(\beta )$, and hence
$$\mathrm{wr}\bigl( r(\alpha )\bigr) =\mathrm{wr}\bigl( i(\beta )\bigr) =-\mathrm{wr}_{+-}( \beta )$$ 
by Lemma~\ref{4-1}(3). 
Since $\beta =\frac{s}{q}=[0, a_n, \ldots , a_2, a_1]$ is of type $\frac{1}{0}$ and 
$N_0(a_n, \ldots , a_1)$ is even, it follows that 
$$\mathrm{wr}( \beta )=-\mathrm{wr}_{+-}\bigl( (ir)(\beta )\bigr)=-\mathrm{wr}_{+-}( \alpha )$$
by Lemma~\ref{4-1}(1). 
Thus, $\mathrm{wr}\bigl( r(\alpha )\bigr) =\mathrm{wr}( \alpha )$. 
Since $n$ is even, 
$$\mathrm{wt}\bigl( r(\alpha )\bigr) 
=-1+(a_n-1)+\sum\limits_{k=2}^n (-1)^{k-1}a_{n-k+1}+2
=a_n-\sum\limits_{k=1}^{n-2} (-1)^{n-k+1}a_k
=\sum\limits_{k=1}^n (-1)^k a_k
=\mathrm{wt}(\alpha ),$$
and hence 
$$\widetilde{\mathrm{wt}}\bigl( r(\alpha )\bigr) 
=-\mathrm{wr}(\alpha )-\mathrm{wt}(\alpha )=\widetilde{\mathrm{wt}}( \alpha ). $$
\par 
Part (3) can be shown by the same manner in the proof of Parts (1) and (2). 
\end{proof} 

\par \bigskip \noindent 
\begin{thm}\label{4-6}
Let $p, q$ be coprime integers satisfying with $0<p<q$, and $( \frac{x}{s}, \frac{y}{r})$ be the pair of parents of 
$\alpha =\frac{p}{q}$. 
Then, 
\begin{equation}\label{eq4.16}
\bigl( \langle \varGamma_{\alpha }\rangle ,\ 
\langle \varGamma_{i(\alpha )}\rangle ,\ 
\langle \varGamma_{r(\alpha )}\rangle ,\ 
\langle \varGamma_{(ir)(\alpha )}\rangle \bigr) =
\bigl( \langle \varGamma_{\alpha }\rangle ,\ 
\overline{\langle \varGamma_{\alpha }\rangle },\ 
\langle \varGamma_{\alpha }\rangle ,\ 
\overline{\langle \varGamma_{\alpha }\rangle }\bigr)
\end{equation}
holds, 
where  $\overline{\langle \varGamma_{\alpha }\rangle }$ denotes the Laurent polynomial obtained from $\langle \varGamma_{\alpha }\rangle $ by replacing $A$ with $A^{-1}$. 
Furthermore, 
\begin{enumerate}
\item[$(1)$] If $q$ is odd, then 
\begin{align*}
V\bigl( i(\alpha )\bigr) &=V\bigl( (ir)(\alpha )\bigr) =\overline{V(\alpha )},\\ 
V\bigl( r(\alpha )\bigr) &=V(\alpha ).
\end{align*}
\item[$(2)$] If $q$ and $x$ are even, then 
\begin{align*}
V\bigl( i(\alpha )\bigr) &=V\bigl( (ir)(\alpha )\bigr) 
=(-A^3)^{-\mathrm{wr}(\alpha )-\mathrm{wr}(i(\alpha ))}\overline{V(\alpha )},\\ 
V\bigl( r(\alpha )\bigr) &=V(\alpha ).
\end{align*}
\item[$(3)$] If $q$ and $y$ are even, then 
\begin{align*}
V\bigl( i(\alpha )\bigr) &=(-A^3)^{-\mathrm{wr}(\alpha )-\mathrm{wr}(i(\alpha ))}\overline{V(\alpha )},\\ 
V\bigl( (ir)(\alpha )\bigr) &=\overline{V(\alpha )},\\ 
V\bigl( r(\alpha )\bigr) &=(-A^3)^{\mathrm{wr}(\alpha )+\mathrm{wr}(i(\alpha ))}V(\alpha ).
\end{align*}
\end{enumerate}
\end{thm} 
\begin{proof} 
The equation \eqref{eq4.16} has already shown in \cite[Theorems 3.12 and 3.15]{Kogiso-Wakui}. 
The rest of all equations can be easily obtained from \eqref{eq4.13}, Lemmas~\ref{4-1} and \ref{4-5}. 
\end{proof} 

\begin{rem}\label{4-7}
In the case where $q$ is even, we may consider the Jones polynomial $V_{+-}(\alpha )$ of the rational link $D_{+-}(T(\alpha ))$. 
By Lemma~\ref{4-1}(3), 
$D_{+-}(T(\alpha ))$ is isotopic to $D_{--}\bigl( \overline{T\bigl( i(\alpha )\bigr)}\bigr)$ as an oriented link. 
This implies that 
\begin{equation}
V_{+-}(\alpha )=\overline{V(i(\alpha ))}.
\end{equation}
\end{rem}

\begin{cor}\label{4-8}
For a Conway-Coxeter frieze $\varGamma$ of zigzag-type, 
we choose a rational number $\alpha $ in the open interval $(0,1)$ such that $\varGamma =\varGamma _{\alpha}$, and $( \frac{x}{s}, \frac{y}{r})$ be the pair of parents of $\alpha =\frac{p}{q}$. 
Define the equivalence class $V(\varGamma )$ by
\begin{equation}
V(\varGamma ):=\begin{cases}
V(\alpha )\equiv \overline{V(\alpha )} &  \text{if $q$ is odd}, \\ 
 V(\alpha )\equiv V(i(\alpha )) & 
 \text{if $q$ and $x$ are even}, \\ 
V(\alpha )\equiv V(i(\alpha ))\equiv \overline{V(i(\alpha ))}\equiv \overline{V(\alpha )}  & \text{if $q$ and $y$ are even}. 
\end{cases}
\end{equation}
Then $V(\varGamma )$ is well-defined. 
We treat $V(\varGamma )$ as a Laurent polynomial in variable $t^{\frac{1}{2}}$ by substituting $t=A^{-4}$, 
and call it the Jones polynomial of the Conway-Coxeter frieze $\varGamma $. 
\end{cor}

\par \medskip \noindent 
\begin{exam}
\begin{enumerate}
\item[$(1)$] When $\alpha =\frac{1}{4}$, by Lemma~\ref{1-2},  
$i\bigl( \frac{1}{4}\bigr) =\frac{3}{4},\ r\bigl( \frac{1}{4}\bigr) =\frac{1}{4},\ (ir)\bigl( \frac{1}{4}\bigr) =\frac{3}{4}$. 
Since the pair of parents of $\alpha$ is $(\frac{0}{1}, \frac{1}{3})$, by Theorem~\ref{4-6}(2), we see that 
$$\Bigl( V(\alpha ),\  V\bigl(i(\alpha )\bigr) ,\ V\bigl(r(\alpha )\bigr) ,\ V\bigl((ir)(\alpha )\bigr) \Bigr) 
=\Bigl( V(\alpha ),\ \overline{V_{+-}(\alpha ) },\ V(\alpha ) ,\ \overline{V_{+-}(\alpha )} \Bigr) .$$ 
Since $\alpha =[0, 4]$ and $i(\alpha )=[0,1,3]$, $\mathrm{wr}(\alpha )=\mathrm{wr}(i(\alpha ))=4$. 
So, $\mathrm{wr}(\alpha )+\mathrm{wr}(i(\alpha ))=8$ and 
$V(\frac{1}{4})=t^{\frac{3}{2}}(-t^3-t+1-t^{-1})$. 
Thus 
$$V(\varGamma _{\frac{1}{4}})\equiv t^{\frac{3}{2}}(-t^3-t+1-t^{-1})\equiv t^{\frac{9}{2}}(-t^{-3}-t^{-1}+1-t). $$
Indeed, $V(\frac{3}{4})=t^{\frac{9}{2}}(-t^{-3}-t^{-1}+1-t)=t^6\overline{V(\frac{1}{4})}$. 
\item[$(2)$]  When $\alpha =\frac{3}{10}$, by Lemma~\ref{1-2},  
$i\bigl( \frac{3}{10}\bigr) =\frac{7}{10},\ r\bigl( \frac{3}{10}\bigr) =\frac{7}{10},\ (ir)\bigl( \frac{3}{10}\bigr) =\frac{3}{10}$. 
Since the pair of parents of $\alpha$ is $(\frac{2}{7}, \frac{1}{3})$, by Theorem~\ref{4-6}(2), we see that 
$$\Bigl( V(\alpha ),\  V\bigl(i(\alpha )\bigr) ,\ V\bigl(r(\alpha )\bigr) ,\ V\bigl((ir)(\alpha )\bigr) \Bigr) 
=\Bigl( V(\alpha ),\ \overline{V_{+-}(\alpha ) },\ V(\alpha ) ,\ \overline{V_{+-}(\alpha )} \Bigr) .$$ 
Since $\alpha =[0, 3, 3]$ and $i(\alpha )=[0, 1, 2, 3]$, $\mathrm{wr}(\alpha )=\mathrm{wr}(i(\alpha ))=6$. 
So, $\mathrm{wr}(\alpha )+\mathrm{wr}(i(\alpha ))=12$ and 
$V\bigl(\frac{3}{10}\bigr) =t^{\frac{9}{2}}(-t^3+t^2-2t+2-2t^{-1}+t^{-2}-t^{-3})$. 
In this case $V(\frac{7}{10})=t^9\overline{V(\frac{3}{10})}=V(\frac{3}{10})$, and 
hence 
$$V(\varGamma _{\frac{3}{10}})\equiv t^{\frac{9}{2}}(-t^3+t^2-2t+2-2t^{-1}+t^{-2}-t^{-3}). $$
The result $V(\frac{7}{10})=V(\frac{3}{10})$ is confirmed by $3\cdot 7\equiv 1\ (\mathrm{mod}\ 2\cdot 10)$. 
Because, by Schubert's classification theorem for the rational links with orientation \cite{Schubert}, 
the congruent equation implies that two oriented links $D(T(\frac{3}{10}))$ and $D(T(\frac{7}{10}))$ are isotopic. 
\item[$(3)$] When $\alpha =\frac{3}{14}$, by Lemma~\ref{1-2},  
$i\bigl( \frac{3}{14}\bigr) =\frac{11}{14},\ r\bigl( \frac{3}{14}\bigr) =\frac{5}{14},\ (ir)\bigl( \frac{3}{14}\bigr) =\frac{9}{14}$. 
Since the pair of parents of $\alpha$ is $(\frac{1}{5}, \frac{2}{9})$, by Theorem~\ref{4-6}(2), we see that 
$$\Bigl( V(\alpha ),\  V\bigl(i(\alpha )\bigr) ,\ V\bigl(r(\alpha )\bigr) ,\ V\bigl((ir)(\alpha )\bigr) \Bigr) 
=\Bigl( V(\alpha ),\ \overline{V_{+-}(\alpha ) },\ V_{+-}(\alpha ) ,\ \overline{V(\alpha )} \Bigr) .$$ 
Since $\alpha =[0, 4, 1, 2]$ and $i(\alpha )=[0, 1, 3, 1, 2]$, $\mathrm{wr}(\alpha )=1,\ \mathrm{wr}(i(\alpha ))=3$. 
So, $\mathrm{wr}(\alpha )+\mathrm{wr}(i(\alpha ))=4$. 
By computation we have 
\begin{align*}
V\Bigl(\frac{3}{14}\Bigr) &=t^{-\frac{3}{2}}(-t^5+t^4-2t^3+2t^2-3t+2-2t^{-1}+t^{-2}), \\ 
V\Bigl(\frac{11}{14}\Bigr) &=t^{\frac{9}{2}}(t^2-2t+2-3t^{-1}+2t^{-2}-2t^{-3}+t^{-4}-t^{-5}). 
\end{align*}
Thus 
$$V\Bigl(\frac{11}{14}\Bigr) =t^{3}\overline{V\Bigl(\frac{3}{14}\Bigr) },\quad 
V\Bigl(\frac{5}{14}\Bigr) =t^{-3}V\Bigl(\frac{3}{14}\Bigr) , \quad 
V\Bigl(\frac{9}{14}\Bigr) =\overline{V\Bigl(\frac{3}{14}\Bigr) }, $$
and 
$$V(\varGamma _{\frac{3}{14}})
\equiv V\Bigl(\frac{3}{14}\Bigr) 
\equiv t^3\overline{V\Bigl(\frac{3}{14}\Bigr) }
\equiv t^{-3}V\Bigl(\frac{3}{14}\Bigr) 
\equiv \overline{V\Bigl(\frac{3}{14}\Bigr) }.$$
\end{enumerate}
\end{exam} 

\par \medskip 
Via the CCFs of zigzag type, we can recognize the following phenomena on the Jones polynomials for rational links. 

\begin{rem}
It is known that there are four pairs of rational knots with less than or equal to $12$ crossings such that their Jones polynomials are the same up to replacing $t$ with $t^{-1}$. 
The pairs are given as follows. 
\begin{enumerate}
\item[$(1)$] $\{ D(T(\frac{29}{49})),\ D(T(\frac{36}{49})) \}$, 
\item[$(2)$] $\{ D(T(\frac{19}{81})),\ D(T(\frac{37}{81}))\}$, 
\item[$(3)$] $\{ D(T(\frac{32}{121})),\ D(T(\frac{43}{121}))\}$, 
\item[$(4)$] $\{ D(T(\frac{64}{147})),\ D(T(\frac{104}{147})) \}$. 
\end{enumerate}

Viewing the corresponding CCFs, we notice that any pair $\{ D(T(\alpha )),\ D(T(\beta ))\}$ of them  has a common characteristic such as $n(\alpha )-n((ir)(\alpha ))=n(r(\alpha ))-n(i(\alpha ))=n(\beta )-n((ir)(\beta ))=n(r(\beta ))-n(i(\beta ))=\pm 2$, and 
$n(\alpha )-n(r(\alpha ))$, $n(\beta )-n(r(\beta ))$ can be divided by any prime factor of $d(\alpha )=d(\beta )$, 
where $n(\alpha )$ and $d(\alpha )$ stand for the numerator and the denominator of $\alpha$. 

\par \medskip \centerline 
{$\begin{matrix}\begin{array}{ccc}
29  &     & 22  \\[0.15cm]
  & 49 &      \\[0.15cm]
27 &      &  20  
\end{array}
\end{matrix}$ \ $\left  | \vbox to 30pt{ }\right. $\ 
$\begin{matrix}\begin{array}{ccc}
36  &     & 15 \\[0.15cm]
     & 49 &      \\[0.15cm]
34 &      &  13
\end{array}
\end{matrix}$\ $\left  | \vbox to 30pt{ }\right. $\ 
$\begin{matrix}\begin{array}{ccc}
19 &     & 64\\[0.15cm]
     & 81 &      \\[0.15cm]
17 &      &  62
\end{array}
\end{matrix}$\ $\left  | \vbox to 30pt{ }\right. $\ 
$\begin{matrix}\begin{array}{ccc}
37  &     & 46\\[0.15cm]
     & 81 &      \\[0.15cm]
35 &      & 44
\end{array}
\end{matrix}$}
\par \medskip \centerline 
{$\begin{matrix}\begin{array}{ccc}
32  &     & 87  \\[0.15cm]
  & 121 &      \\[0.15cm]
34 &      &  89  
\end{array}
\end{matrix}$ \ $\left  | \vbox to 30pt{ }\right. $\ 
$\begin{matrix}\begin{array}{ccc}
43  &     & 76 \\[0.15cm]
     & 121 &      \\[0.15cm]
45 &      &  78
\end{array}
\end{matrix}$\ $\left  | \vbox to 30pt{ }\right. $\ 
$\begin{matrix}\begin{array}{ccc}
64  &     & 85\\[0.15cm]
     & 147 &      \\[0.15cm]
62 &      &  83
\end{array}
\end{matrix}$\ $\left  | \vbox to 30pt{ }\right. $\ 
$\begin{matrix}\begin{array}{ccc}
106  &     & 43\\[0.15cm]
     & 147 &      \\[0.15cm]
104 &      & 41
\end{array}
\end{matrix}$}
\par \medskip 
Further development will be appeared in a forthcoming paper. 
\end{rem}

\section{A recurrence formula of the writhe of a rational link diagram in terms of continued fractions}
\par 
Let us consider a rational number $\alpha $ in $(0,1)$, and its Yamada's ancestor triangle $\text{YAT}(\alpha )$ (see \cite{Kogiso-Wakui, Yamada-Proceeding} for the precise definition and details). 
We write $\alpha $ in the continued fraction form $\alpha =[0, a_1, \ldots , a_n]$. 
Then, there is a unique downward path in $\text{YAT}(\alpha )$, which is started from $0$ to $\alpha $, and 
is passing through the vertices 
$$[0],\ [0, a_1],\ [0, a_1, a_2],\ \ldots ,\ [0, a_1, a_2, \ldots , a_n].$$
We call the path the \textit{continued fraction path} associated with $\alpha $. 
By the continued fraction path $\text{YAT}(\alpha )$ is divided into $n$ triangles, which are named as $\Delta_1, \ldots, \Delta_n$ from the top. 
We note that if $j$ is odd, then the vertex corresponding to $[0, a_1, \ldots , a_j]$ is on the right oblique line, and otherwise it is on the left. 

\par \smallskip 
\begin{exam} 
For $\alpha =\frac{3}{8}=[0,2,1,2]$, $[0]=\frac{0}{1},\ [0,2]=\frac{1}{2},\ [0,2,1]=\frac{1}{3}$. 

\begin{figure}[htbp]
\centering\includegraphics[height=4cm]{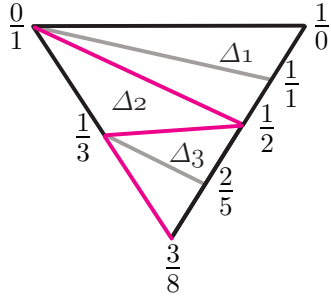}
\caption{the continued fraction path of $\frac{3}{8}$}
\label{fig4}
\end{figure}

Thus, the continued fraction path of 
$\frac{3}{8}$ is the path $\frac{0}{1}\to \frac{1}{2}\to \frac{1}{3}\to \frac{3}{8}$. See Figure~\ref{fig4}. 
\end{exam}

\par \smallskip 
\begin{thm}\label{5-2}
For each $j\in \{ 1, \ldots , n\}$ let $t_{\alpha}(\Delta _j)$ be the sign of $\Delta _j$ defined in Remark~\ref{4-4}, and set $\alpha _j=[0, a_1, \ldots , a_j]$. 
\par 
In the case where $n=2$, 
$$t_{\alpha }(\Delta _2)=(-1)^{a_1},\ t_{\alpha }(\Delta _1)=(-1)^{a_1a_2+a_2+1}. $$
In the case where $n\geq 3$, 
\begin{enumerate}
\item[$(1)$] if $\alpha _{n-2}$ is $\frac{1}{1}$-type and $\alpha _{n-1}$ is $\frac{1}{0}$-type, then 
$$t_{\alpha}(\Delta_j)=\begin{cases}
(-1)^{d(\alpha _j)(a_n-1)}t_{\alpha_{n-1}}(\Delta_j)\  & (j=1,2,\ldots , n-1),\\ 
t_{\alpha_{n-1}}(\Delta _{n-1})\  & (j=n).
\end{cases}$$
\item[$(2)$] if $\alpha _{n-2}$ is $\frac{0}{1}$-type and $\alpha _{n-1}$ is $\frac{1}{0}$-type, then
$$t_{\alpha}(\Delta_j)=\begin{cases}
(-1)^{d(\alpha _j)a_n}t_{\alpha_{n-1}}(\Delta_j)\  & (j=1,2,\ldots , n-1),\\ 
-t_{\alpha_{n-1}}(\Delta _{n-1})\  & (j=n).
\end{cases}$$
\item[$(3)$] if $\alpha _{n-2}$ is $\frac{1}{0}$-type and $\alpha _{n-1}$ is $\frac{0}{1}$-type, then
\begin{enumerate}
\item[$\bullet$] if $a_n$ is even, then 
$$t_{\alpha}(\Delta_j)=\begin{cases}
t_{\alpha_{n-1}}(\Delta_j)\  & (j=1,2,\ldots , n-1),\\ 
-t_{\alpha_{n-1}}(\Delta _{n-1})\  & (j=n).
\end{cases}$$
\item[$\bullet$] if $a_n$ is odd, then
$$t_{\alpha}(\Delta_j)=\begin{cases}
(-1)^{d(\alpha _j)}t_{\alpha_{n-2}}(\Delta_j)\  & (j=1,2,\ldots , n-2),\\ 
(-1)^{a_{n-1}-1}t_{\alpha_{n-2}}(\Delta_{n-2})\  & (j=n-1, n).
\end{cases}$$
\end{enumerate}
\item[$(4)$] if $\alpha _{n-2}$ is $\frac{1}{1}$-type and $\alpha _{n-1}$ is $\frac{0}{1}$-type, then 
\begin{enumerate}
\item[$\bullet$] if $a_n$ is even, then 
$$t_{\alpha}(\Delta_j)=\begin{cases}
t_{\alpha_{n-2}}(\Delta_j)\  & (j=1,2,\ldots , n-2),\\ 
(-1)^{a_{n-1}}t_{\alpha_{n-2}}(\Delta_{l-2})\  & (j=n-1, n).
\end{cases}$$
\item[$\bullet$] if $a_n$ is odd, then 
$$t_{\alpha}(\Delta_j)=\begin{cases}
t_{\alpha_{n-1}}(\Delta_j)\  & (j=1,2,\ldots , n-1),\\ 
t_{\alpha_{n-1}}(\Delta _{n-1})\  & (j=n).
\end{cases}$$
\end{enumerate}
\item[$(5)$] if $\alpha _{n-2}$ is $\frac{0}{1}$-type and $\alpha _{n-1}$ is $\frac{1}{1}$-type, then 
$$t_{\alpha}(\Delta_j)=\begin{cases}
t_{\alpha_{n-2}}(\Delta_j)\  & (j=1,2,\ldots , n-2),\\ 
(-1)^{a_{n-1}}t_{\alpha_{n-2}}(\Delta_{l-2})\  & (j=n-1, n).
\end{cases}$$
\item[$(6)$] if $\alpha _{n-2}$ is $\frac{1}{0}$-type and $\alpha _{n-1}$ is $\frac{1}{1}$-type, then 
$$t_{\alpha}(\Delta_j)=\begin{cases}
t_{\alpha_{n-2}}(\Delta_j)\  & (j=1,2,\ldots , n-2),\\ 
(-1)^{a_{n-1}}t_{\alpha_{n-2}}(\Delta_{n-2})\  & (j=n-1, n).
\end{cases}$$
\end{enumerate} 
Here, $d(\alpha _j)\in \{0,1\}$ is the denominator of the type of $\alpha _j$. 
\end{thm}

To prove the theorem let us recall the definition of a Seifert path, which is introduced by Nagai and Terashima \cite{Nagai_Terashima}. 
Let $\alpha $ be a rational number in $(0,1)$. 
Every vertex in the Yamada's ancestor triangle $\text{YAT}(\alpha )$ is one of the $\frac{1}{1}, \frac{1}{0}, \frac{0}{1}$-types. 
A \textit{Seifert path} of $\alpha $ is a downward path in $\text{YAT}(\alpha )$, which is started from $\frac{1}{0}$ to $\alpha $ satisfying the following condition: 
The end points of any edge in the path consist of $\frac{1}{1}$- and $\frac{1}{0}$-types, or consist of $\frac{1}{0}$- and $\frac{0}{1}$-types. 
If the denominator of $\alpha $ is odd, then a Seifert path is uniquely determined. 
We denote the Seifert path by $\gamma _{\alpha }$. 
If the denominator of $\alpha $ is even, namely $\alpha $ is of type $\frac{1}{0}$, then there are exactly two Seifert paths. 
In this case we denote by $\gamma _{\alpha }$ the Seifert path whose vertices consist of $\frac{1}{0}$- and $\frac{0}{1}$-types, 
and denote by $\gamma _{\alpha }^{\prime}$ the remaining Seifert path. 
As a similar to $t_{\alpha }(\Delta_j)$, we define a sign $t_{\alpha }^{\prime}(\Delta_j)$ by the following inductive rules. 
\begin{enumerate}
\item[$\bullet$] at first, set $t_{\alpha }^{\prime}(\Delta_1):=-1$, and 
\item[$\bullet$] after $t_{\alpha }^{\prime}(\Delta_j)$ is defined, set 
$$t_{\alpha }^{\prime}(\Delta_{j+1}):=\begin{cases}
t_{\alpha }^{\prime}(\Delta_j) & \text{if there is no edge in $\gamma _{\alpha}^{\prime}$ between $\Delta _j$ and $\Delta _{j+1}$}, \\ 
-t_{\alpha }^{\prime}(\Delta_j) & \text{otherwise}. 
\end{cases}$$
\end{enumerate}

We remark that 
$t_{\alpha }^{\prime}(\Delta_j)=\epsilon _jt_{\alpha }(\Delta_j)$ holds for $j=1,2,\ldots ,n$, where 
$$\epsilon _1=1,\quad \ \  
\epsilon _j=\begin{cases}
1 & \text{if the denominator of $\alpha_{j-1}$ is odd},\\ 
-1 & \text{otherwise}. 
\end{cases}$$

\medskip \noindent 
{\bf Proof of Theorem~\ref{5-2}}
\par 
It can be easily verified in the case where $n=2$. 
So, we consider the case where $n\geq 3$. 
In this case, the statement can be shown by case-by-case argument. 
We only demonstrate the proof of Part (1) since other cases are verified by a quite similar argument. 
\par 
Consider the case where $a_n$ is even. 
Then the Seifert path $\gamma_{\alpha}$ is obtained by connecting $\gamma_{\alpha_{n-1}}^{\prime}$ with the edges between $\alpha_{n-1}$ and $\alpha$.  

\begin{align*}
\includegraphics[height=4.5cm]{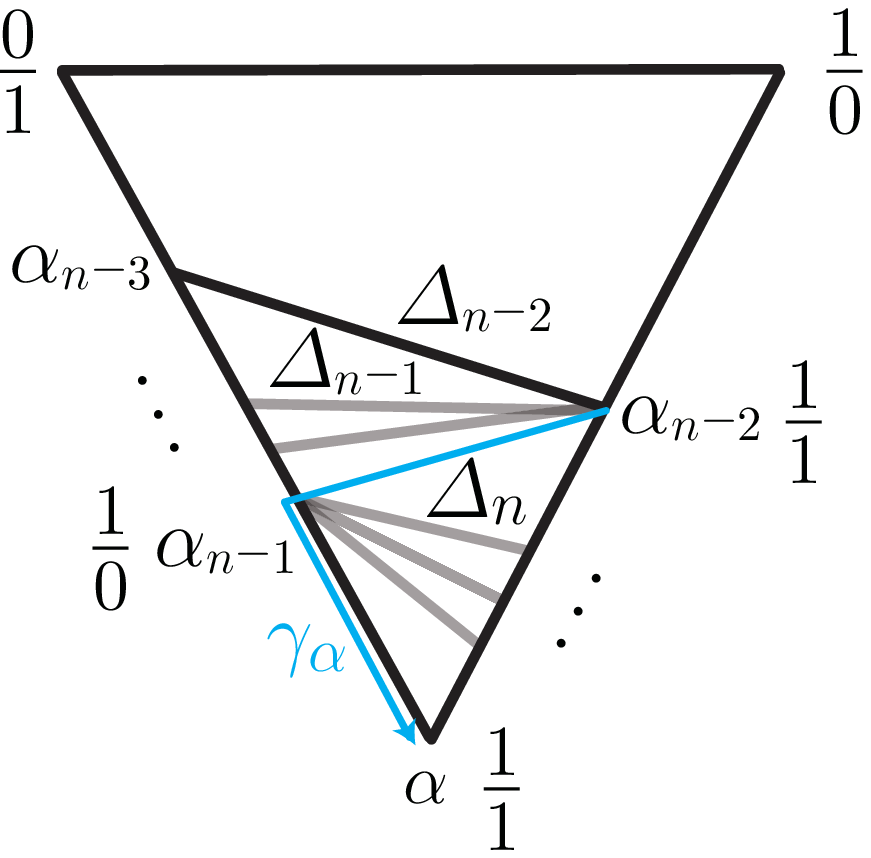} & \hspace{1cm}
\includegraphics[height=4.5cm]{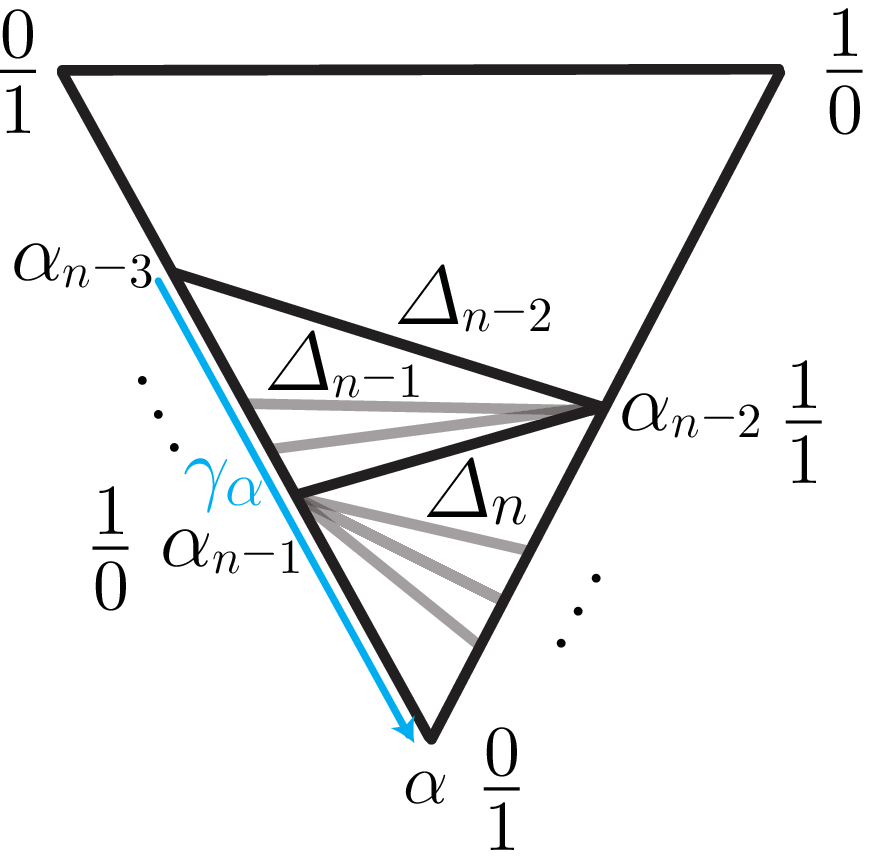} \\ 
\hspace{1cm} \text{$c_l$ is even} \hspace{1.5cm} & \hspace{2.5cm} \text{$c_l$ is odd}
\end{align*}

Since there is an edge in $\gamma_{\alpha}$ between $\Delta_{n-1}$ and $\Delta_n$, 
we have $t_{\alpha }(\Delta _n)=-t_{\alpha }(\Delta _{n-1})$. 
Moreover, for all $j=1,2,\ldots , n-2$, we see that 
$$t_{\alpha}(\Delta_j)=-t_{\alpha}(\Delta_{j+1})\ \ \Longleftrightarrow \ \ t_{\alpha_{n-1}}^{\prime}(\Delta_j)=-t_{\alpha_{n-1}}^{\prime}(\Delta_{j+1}).$$
Since $\alpha$ is $\frac{1}{1}$-type whereas $\alpha_{n-1}$ is $\frac{1}{0}$-type, 
$t_{\alpha}(\Delta_1)=1$ and $t_{\alpha_{n-1}}^{\prime}(\Delta_1)=-1$. 
Thus, 
$$t_{\alpha}(\Delta_j)=-t_{\alpha_{n-1}}^{\prime}(\Delta_j)=-\epsilon_jt_{\alpha_{n-1}}(\Delta_j)
$$
for all $j=1,2,\ldots , n-1$. 
Since $\alpha _{n-2}$ is $\frac{1}{1}$-type, 
we have 
$t_{\alpha}(\Delta_{n-1})=-t_{\alpha_{n-1}}(\Delta _{n-1})$, and hence 
$t_{\alpha }(\Delta _n)=-t_{\alpha }(\Delta _{n-1})=t_{\alpha_{n-1}}(\Delta _{n-1})$. 
It follows that 
$$t_{\alpha}(\Delta_j)=\begin{cases}
-\epsilon_jt_{\alpha_{n-1}}(\Delta_j)\  & (j=1,2,\ldots , n-1),\\ 
t_{\alpha_{n-1}}(\Delta _{n-1})\  & (j=n).
\end{cases}$$
\par 
Consider the case where $a_n$ is odd. 
Then $\gamma_{\alpha}$ is obtained by connecting $\gamma_{\alpha_{n-1}}$ with the edges between $\alpha_{n-1}$ and $\alpha$. 
In this case, since there is no edge in $\gamma_{\alpha}$ between $\Delta_{n-1}$ and $\Delta_n$, we have $t_{\alpha }(\Delta _n)=t_{\alpha }(\Delta _{n-1})$. 
Since $\alpha$ is $\frac{0}{1}$-type and $\alpha_{n-1}$ is $\frac{1}{0}$-type, 
it follows that 
$t_{\alpha}(\Delta_1)=-1=t_{\alpha_{n-1}}(\Delta_1)$. 
So, by the same argument above, we see that 
$$t_{\alpha}(\Delta_j)=t_{\alpha_{n-1}}(\Delta_j)
$$
for all $j=1,2,\ldots , n-1$. 
In particular, 
$t_{\alpha}(\Delta_{n-1})=t_{\alpha_{n-1}}(\Delta _{n-1})$. Thus 
$t_{\alpha }(\Delta _n)=t_{\alpha }(\Delta _{n-1})=t_{\alpha_{n-1}}(\Delta _{n-1})$, and therefore, 
$$t_{\alpha}(\Delta_j)=\begin{cases}
t_{\alpha_{n-1}}(\Delta_j)\  & (j=1,2,\ldots , n-1),\\ 
t_{\alpha_{n-1}}(\Delta _{n-1})\  & (j=n).
\end{cases}$$
\qed 

\par \medskip 
By induction argument we have the following corollary from Theorem~\ref{5-2}. 

\par \smallskip 
\begin{cor}\label{5-3}
Under the same notation with Theorem~\ref{5-2}, 
the sign $t_{\alpha }(\Delta _n)$ of the top triangle $\Delta _n$ is given by the formula: 
$$t_{\alpha }(\Delta _n)=(-1)^{(d(\alpha _n)+1)n(\alpha_{n-1})+d(\alpha_n)n(\alpha _{n-1})+n}, $$
where, $d(\alpha _j), n(\alpha _j)\in \{0,1\}$ are the denominator and the numerator of the type of $\alpha _j$, respectively. 
\end{cor}

\par \smallskip 
\begin{exam}\label{5-4}
\begin{enumerate}
\item[$(1)$] Consider the case where $\alpha =\frac{3}{8}=[0, 2, 1, 2]$. 
We set 
$$\alpha _1:=[0, 2]=\dfrac{1}{2},\qquad \alpha _2:=[0, 2, 1]=\dfrac{1}{3}\qquad \alpha _3:=[0, 2, 1, 2]=\dfrac{3}{8},$$
and let $\{ \Delta _1, \Delta _2, \Delta _3\} $ be the corresponding triangle sequence for $\alpha $. 
Since $\alpha _1$ is $\frac{1}{0}$-type and $\alpha _2$ is $\frac{1}{1}$-type, it follows from Theorem~\ref{5-2}(6) that 
\begin{align*}
t_{\alpha }(\Delta _3)&=t_{\alpha }(\Delta _2)=(-1)^{2-1}t_{\alpha _1}(\Delta _1)=-t_{\alpha _1}(\Delta _1),\\ 
t_{\alpha }(\Delta _1)&=t_{\alpha _1}(\Delta _1). 
\end{align*}
By Theorem~\ref{5-2} again, we have 
$t_{\alpha _1}(\Delta _1)=(-1)^{2+1}=-1$, and therefore 
$$t_{\alpha }(\Delta _3)=t_{\alpha }(\Delta _2)=1,\ t_{\alpha }(\Delta _1)=-1.$$
By using the formula \eqref{eq4.14} we see that the writhe of the oriented diagram $D(T(\frac{3}{8}))$ is given by 
$\mathrm{wr}\bigl(\frac{3}{8}\bigr) =-\bigl( (-1)\cdot 2+1\cdot 1+1\cdot 2\bigr) =-1$. 
\item[$(2)$] Consider the case where $\alpha =\frac{8}{11}=[0, 1, 2, 1, 2]$. 
We set 
$$\alpha _1:=[0, 1]=\dfrac{1}{1},\quad \alpha _2:=[0, 1, 2]=\dfrac{2}{3}\quad \alpha _3:=[0, 1, 2, 1]=\dfrac{3}{4}\quad \alpha _4:=[0, 1, 2, 1, 2]=\dfrac{8}{11},$$
and let $\{ \Delta _1, \Delta _2, \Delta _3, \Delta _4\} $ be the corresponding triangle sequence for $\alpha$. 
Since $\alpha _2$ is $\frac{0}{1}$-type and $\alpha _3$ is $\frac{1}{0}$-type, it follows from Theorem~\ref{5-2}(2) that 
\begin{align*}
t_{\alpha }(\Delta _4)&=-t_{\alpha _3}(\Delta _3),\\ 
t_{\alpha }(\Delta _j)&=(-\epsilon_j)^2 t_{\alpha _3}(\Delta _j)=t_{\alpha _3}(\Delta _j)
\end{align*}
for $j=1,2,3$. 
In addition, $\alpha _1$ is $\frac{1}{1}$-type and $c_3=1$ is odd. 
Thus, applying Theorem~\ref{5-2}(4) we have 
\begin{align*}
t_{\alpha _3}(\Delta _3)&=t_{\alpha _2}(\Delta _2),\\ 
t_{\alpha _3}(\Delta _j)&=t_{\alpha _2}(\Delta _j)
\end{align*}
for $j=1,2$. Since 
$t_{\alpha _2}(\Delta _2)=(-1)^1=-1,\ t_{\alpha _2}(\Delta _1)=(-1)^{2+3}=-1$, we
see that 
\begin{align*}
t_{\alpha }(\Delta _4)&=-t_{\alpha _3}(\Delta _3)=-t_{\alpha _2}(\Delta _2)=1,\\ 
t_{\alpha }(\Delta _3)&=t_{\alpha _3}(\Delta _3)=t_{\alpha _2}(\Delta _2)=-1,\\ 
t_{\alpha }(\Delta _2)&=t_{\alpha _3}(\Delta _2)=t_{\alpha _2}(\Delta _2)=-1,\\ 
t_{\alpha }(\Delta _1)&=t_{\alpha _3}(\Delta _1)=t_{\alpha _2}(\Delta _1)=-1. 
\end{align*}
By using the formula \eqref{eq4.14} we see that the writhe of the oriented diagram $D(T(\frac{8}{11}))$ is given by 
$\mathrm{wr}\bigl(\frac{8}{11}\bigr) =2$. 
\end{enumerate}
\end{exam}

\par \medskip \noindent 
{\bf Acknowledgments}. 
We would like to thank Professors Mikami Hirasawa, Makoto Sakuma, Yuji Terashima for many helpful comments.


\begin{thebibliography}{99}

\bibitem{Conway}
J. H. Conway,
\textit{An enumeration of knots and links, and some of their algebraic properties}, 
in {\it Proceedings of the conference on computational problems in abstract algebra held at Oxford 1967}, edited by J. Leech,  (Pergamon Press, 1970), pp.~329--358. 

\bibitem{CoCo1}
J.H. Conway, H.S.M. Coxeter, 
\textit{Triangulated polygons and frieze patterns}, 
Math. Gaz. {\bf 57} (1973), no. 400, 87--94.

\bibitem{CoCo2}
J.H. Conway, H.S.M. Coxeter, 
\textit{Triangulated polygons and frieze patterns II}, 
Math. Gaz. {\bf 57} (1973), no. 401, 175--183. 

\bibitem{Coxeter}
H.S.M. Coxeter,
\textit{Frieze patterns},
Acta Arith. {\bf 18} (1971), 297--310. 

\bibitem{Cromwell}
P. Cromwell,
\textit{Knots and links}, 
Cambridge University Press, 2004.

\bibitem{HO}
A. Hatcher and U. Ortel, 
\textit{Boundary slopes for Montesinos knots}, 
Topology {\bf 28} (1989), 453--480. 

\bibitem{Kau-Topology}
L.H. Kauffman, 
\textit{State models and the Jones polynomial}, 
Topology {\bf 26} (1987), 395--407. 

\bibitem{Kogiso-Wakui}
T. Kogiso and M. Wakui, 
\textit{A bridge between Conway-Coxeter Friezes and rational tangles through the Kauffman bracket polynomials}, 
J. Knot Theory Ramifications (2019), 1950083, 40pp. 

\bibitem{LeeSchiffler}
K. Lee and R. Schiffler,
\textit{Cluster algebras and Jones polynomials}, 
Selecta Math. (N.S.) {\bf 25} (2019), Paper No.58, 41pp.

\bibitem{M-G}
S. Morier-Genoud, 
\textit{Coxeter's frieze patterns at the crossroads of algebra, geometry and combinatorics},
Bull. London Math. Soc. {\bf 47} (2015), 895--938. 

\bibitem{M-GO2}
S. Morier-Genoud and V. Ovsienko, 
\textit{$q$-deformed rationals and $q$-continued fractions},
Forum Math. Sigma {\bf 8} (2020), e13, 55pp.

\bibitem{Murasugi}
K. Murasugi, 
\textit{Knot theory and its applications}, 
translated from the 1993 Japanese original by Bohdan Kurpita, Birkh\"{a}user, Boston, MA, 1996. 

\bibitem{Nagai_Terashima}
W. Nagai and Y. Terashima, 
\textit{Cluster variables, ancestral triangles and Alexander polynomials}, 
Adv. Math. {\bf 363} (2020), 106965, 37pp.

\bibitem{Schubert}
H. Schubert,
\textit{Knoten mit zwei Br\"{u}cken}, Math. Zeit. {\bf 66} (1956), 133--170. 

\bibitem{Yamada-Proceeding}
S. Yamada, 
\textit{Jones polynomial of two-bridge knots (Ni-hashi musubime no Jones takoushiki)}, in Japanese, 
in Proceedings of \lq\lq Musubime no shomondai to saikin no seika", 1996, 92--96. 




\end{thebibliography}
\end{document}